\documentclass[11pt]{amsart}
\usepackage{mathrsfs,latexsym,amsfonts,amssymb}
\usepackage{hyperref}
\usepackage[utf8]{inputenc}
\newtheorem{theorem}{Theorem}[section]
\newtheorem{lemma}[theorem]{Lemma}
\newtheorem{corollary}[theorem]{Corollary}
\newtheorem{question}[theorem]{Question}
\newtheorem{example}[theorem]{Example}
\theoremstyle{definition}
\newtheorem{definition}[theorem]{Definition}
\newtheorem{proposition}[theorem]{Proposition}
\theoremstyle{remark}

\begin{document}

\title[Hyperspace of finite unions of convergent sequences]
{Hyperspace of finite unions of convergent sequences}

\author{JingLing Lin}
\address{(JingLing Lin): School of mathematics and statistics,
Minnan Normal University, Zhangzhou 363000, P. R. China}
\email{jinglinglin1995@163.com}

\author{Fucai Lin*}
\address{(Fucai Lin): 1. School of mathematics and statistics,
Minnan Normal University, Zhangzhou 363000, P. R. China; 2. Fujian Key Laboratory of Granular Computing and Applications,
Minnan Normal University, Zhangzhou 363000, P. R. China}
\email{linfucai@mnnu.edu.cn; linfucai2008@aliyun.com}

\author{Chuan Liu}
\address{(Chuan Liu): Department of Mathematics,
Ohio University Zanesville Campus, Zanesville, OH 43701, USA}
\email{liuc1@ohio.edu}

\dedicatory{Dedicated to professor Jinjin Li on the occasion of his 60 years anniversary.}

\thanks{The second author is  supported by the Key Program of the Natural Science Foundation of Fujian Province (No: 2020J02043), the NSFC (No. 11571158), the Program for New Century Excellent Talents in Fujian Province University, the Institute of Meteorological Big Data-Digital Fujian and Fujian Key Laboratory of Data Science and Statistics. The first author is supported by the Young and middle-aged project in Fujian Province (No. JAT190397).\\
*corresponding author}


\begin{abstract}
The symbol $\mathcal{S}(X)$ denotes the hyperspace of finite unions of convergent sequences in a
Hausdorff space $X$. This hyperspace is endowed with the Vietoris topology. First of all, we give a characterization of convergent sequence in $\mathcal{S}(X)$. Then we consider some cardinal invariants on $\mathcal{S}(X)$, and compare the character, the pseudocharacter, the $sn$-character, the $so$-character, the network weight and $cs$-network weight of $\mathcal{S}(X)$ with
the corresponding cardinal function of $X$. Moreover, we consider rank $k$-diagonal on $\mathcal{S}(X)$, and give a space $X$ with a rank 2-diagonal such that $S(X)$ does not have any $G_{\delta}$-diagonal. Further, we study the relations of some generalized metric properties of $X$ and its hyperspace $\mathcal{S}(X)$.
Finally, we pose some questions about the hyperspace $\mathcal{S}(X)$.
\end{abstract}

\maketitle
\section{Introduction}
The hyperspaces were originally studied by Vietoris in 1920s. In the study of hyperspaces, ones mainly discussed some special nonempty subsets, such as closed subsets, compact subsets, finite subsets and so on, see \cite{CG, KB, M, TM,TM1,TM1997,SA,IN, LX,TZ}. It is well known that each sequential space is determined by their nontrivial convergent sequences. In 2015, the authors in \cite{GO} first introduced hyperspaces consisting of nontrivial convergent sequences, which is denoted by $S_{c}(X)$. In \cite{DPR}, the authors pointed out that the convergence of sequences is an important tool to determine the topological properties in Hausdorff
spaces, and the study of hyperspaces can provide information about the topological
behavior of the original space and vice versa. In this paper, we mainly consider the hyperspace $\mathcal{S}(X)$ of finite unions of convergent sequences in a Hausdorff space $X$ which contains the hyperspace consisting of non-trivial convergent sequences and the hyperspace consisting of finite subsets.

Given a space $X$, we define its hyperspaces as the following sets:

\smallskip
$2^{X}=\{A\subset X : A\ \mbox{is nonempty and closed in}\ X\}$;

\smallskip
$\mathcal{K}(X)=\{A\in 2^{X}: A\ \mbox{is compact in}\ X\};$

\smallskip
$\mathcal{F}_{n}(X)=\{F\in 2^{X}: F\ \mbox{is a finite subset of}\ X\};$

\smallskip
$\mathcal{S}(X)=\{S\in 2^{X}: S\ \mbox{is the unions of finitely many convergent sequences of}\ X\}.$

Let $\mathcal{P}$ be a family of subsets of a space $X$. For any
$r\in\mathbb{N}$ and $P_1, ..., P_r\in \mathcal{P}$, denote the set $$\{S\in \mathcal{S}(X): S\subset
{\bigcup_{i=1}^r}P_i\ \mbox{and}\ S\cap P_j\neq \emptyset, 1\leq
j\leq r\}$$by $\langle
P_1, ..., P_r\rangle$. If $\mathcal{B}=\{P_1, ..., P_r\}$, then we denote $\langle
P_1, ..., P_r\rangle$ by $\langle\mathcal{B}\rangle$. In particular, if $\mathcal{P}$ be the family of all open subsets of a space $X$, then we just endow
$\mathcal{S}(X)$ with the Vietoris (that is, finite) topology, and a basic open set of $\mathcal{S}(X)$ is of the form
$\langle\mathcal{B}\rangle$ for any $\mathcal{B}=\{P_1, ..., P_r\}\subset\mathcal{P}$ and $r\in\mathbb{N}$.

The paper is organized as follows. In Section 2, we
introduce the necessary notations and terminology which are used for
the rest of the paper. In Section 3, we give a characterization of convergent sequence in $\mathcal{S}(X)$, which plays an important role in this paper.
In Section 4, we mainly discuss some cardinal invariants on $\mathcal{S}(X)$, such as, the character, the pseudocharacter, the $sn$-character, the $so$-character, the network weight and $cs$-network weight, etc. In Section 5, we study the rank $k$-diagonal of $\mathcal{S}(X)$, and gives an example to show that $\mathcal{S}(X)$ does not have any $G_{\delta}$-diagonal even if $X$ has a rank $2$-diagonal. In Section 6, we consider some generalized metric property  on $\mathcal{S}(X)$, such as, $\gamma$-spaces. In Section 7, we pose some questions on $\mathcal{S}(X)$.

\smallskip
\section{Preliminaries}
In this section, we introduce the necessary notations and terminology. Throughout this paper, all topological spaces are assumed to be Hausdorff, unless otherwise is explicitly stated. First of all, let $\mathbb{N}$, $\omega$, $\mathbb{Q}$, $\mathbb{P}$ and $\mathbb{R}$ denote the sets of all positive
  integers, non-negative integers, rational number, irrational number and real numbers, respectively. If $\mathcal{P}$ is a family of subsets of $X$, we denote $\mathcal{P}^{<\omega}$ by the set of all the finite subsets of $\mathcal{P}$. For undefined
  notations and terminology, the reader may refer to
  \cite{E1989}, \cite{G1984} and \cite{linbook}.

\begin{definition}
Let $X$ be a topological space, and let $A\subset X$.

\smallskip
$\bullet$  A subset $U$ of $X$ is called a {\it sequential neighborhood} of $A$ if $A\subset U$ and each sequence converging to some point $x\in A$ is eventually in $U$.

\smallskip
$\bullet$  A subset $P$ of $X$ is called a {\it sequential neighborhood} of $x \in X$, if each sequence converging to $x$ is eventually in $P$.

\smallskip
$\bullet$  A subset $U$ of $X$ is called {\it sequentially open} if $U$ is a sequential neighborhood of each of its points.
\end{definition}

Let $X$ be a space. We say that a sequence $\{A_{n}\}$ consisting of subsets of $X$ converges to a subset $A\subset X$ if for each open set $U$ in $X$ with $A\subset U$ there exists $N\in\mathbb{N}$ such that $A_{n}\subset U$ for any $n>N$.

\begin{definition}
Let $X$ be a space and let $\mathcal{P}$ be a cover of $X$. The family $\mathcal{P}$ is a {\it $CS$-network} of $X$ if, whenever a sequence $\{A_{n}\}_{n\in\mathbb{N}}$ consisting of subsets of $X$ converges to a subset $A\subset X$ and $U$ is an open neighborhood of $A$ in $X$, then there exist $m\in\mathbb{N}$ and $P\in \mathcal{P}$ such that $A\cup\bigcup\{A_{n}:n\geq m\}\subset P\subset U$.
\end{definition}

\begin{definition}
Recall that a subset $H$ of space X is a $G_{\delta}$-set if there is a sequence $\{U_{i}\}_{i\in \mathbb{N}}$ of open sets in $X$ such that $H=\bigcap_{i\in \mathbb{N}}U_{i}$ . A space $X$ is said to have a $G_{\delta}$-diagonal if the diagonal $\Delta_{X}=\{\langle x,x\rangle:x\in X\}$ is a $G_{\delta}$-set in $X^{2}$.
\end{definition}

\begin{definition}\cite{linbook}
A space $X$ is called an \emph{$S_{2}$}-{space} ({\it Arens' space})  if
$$X=\{\infty\}\cup \{x_{n}: n\in \mathbb{N}\}\cup\{x_{n, m}: m, n\in
\omega\}$$ and the topology is defined as follows:

\smallskip
(i) Each $x_{n, m}$ is isolated;

\smallskip
(ii) A basic neighborhood of $x_{n}$ is $\{x_{n}\}\cup\{x_{n, m}: m>k\}$, where $k\in\omega$;

\smallskip
(iii) A basic neighborhood of $\infty$ is $$\{\infty\}\cup (\bigcup\{V_{n}:n>k\})\ \mbox{for some}\ k\in \omega,$$ where $V_{n}$ is a neighborhood of $x_{n}$ for each $n\in\omega$.
\end{definition}

\begin{definition}\cite{linbook}
A space $X$ is called an \emph{ $S_\omega$}-{space} if
$$X=\{x\}\cup\{x_n(m): m\in \omega, n\in
\mathbb{N}\}$$ and the topology is defined as follows:

\smallskip
(a) Each $x_n(m)$ is an isolated point of $X$;

\smallskip
(b) The basic neighborhoods of $x$ is $$\{\{x\}\cup\{x_n(m): n\geq f(m)\}\}, \mbox{where}\ f\in\mathbb{N}^\omega.$$.
\end{definition}

\begin{definition}\cite{linbook}
Let $\mathscr P$ be a cover of a space $X$ such that

\smallskip
(1) $\mathscr P =\bigcup_{x\in X}\mathscr{P}_{x}$;

\smallskip
(2) For each point $x\in X$, if $U, V\in\mathscr{P}_{x}$, we have that $W\subset U\cap V$ for some $W\in \mathscr{P}_{x}$;

\smallskip
(3) For each point $x\in X$ and each open neighborhood $U$ of $x$, there is a $P\in\mathscr P_x$ such that $x\in P \subset U$.

\smallskip
$\bullet$ For any point $x\in X$, if each element of $\mathscr P_x$ is a sequential neighborhood of $x$ in $X$, then $\mathscr P_x$ is called an \emph{sn-network} at point $x$ in $X$. The family $\mathscr P$ is called an \emph{sn-network} for $X$ if each $\mathscr P_x$ is an sn-network at $x$ for each $x\in X$, and $X$ is called \emph{snf-countable} if $X$ has an $sn$-network $\mathscr P$ such that $\mathscr P_x$ is countable for all $x\in X$.

\smallskip
$\bullet$ For any point $x\in X$, if each element of $\mathscr P_x$ is a sequential open neighborhood of $x$ in $X$, then $\mathscr P_x$ is called an \emph{so-network} at point $x$ in $X$. The family $\mathscr P$ is called an \emph{so-network} for $X$ if each $\mathscr P_x$ is an so-network at $x$ for each $x\in X$, and $X$ is called \emph{sof-countable} if $X$ has an $so$-network $\mathscr P$ such that $\mathscr P_x$ is countable for all $x\in X$.
\end{definition}

\begin{definition}\cite{linbook}
Let $\mathscr P =\bigcup_{x\in X}\mathscr{P}_{x}$ be a cover of a space $X$, where each $x\in \cap\mathscr{P}_{x}$.

\smallskip
$\bullet$ For each $x\in X$, if for every sequence $\{x_{n}\}_{n\in\mathbb{N}}$ converging to $x\in U$ with $U$ open in $X$, there exists $P\in \mathscr{P}_{x}$ such that $\{x_{n}\}_{n\in\mathbb{N}}$ is eventually in $P$ and $P\subset U$, then $\mathscr{P}_{x}$ is called a \emph{cs-network} at $x$ in $X$. The family $\mathscr P$ is called
a \emph{cs-network} for $X$ if each $\mathscr{P}_{x}$ is a cs-network at $x$ for each $x\in X$. A space $X$ is called \emph{csf-countable} if $X$ has a
\emph{cs-network} $P$ such that each $\mathscr{P}_{x}$ is countable.

\smallskip
$\bullet$ For each $x\in X$, if for every sequence $\{x_{n}\}$ converges to $x\in U$ with $U$ open in $X$, there is a $P\in\mathcal{P}_{x}$ and a subsequence $\{x_{n_{i}}: i\in \mathbb{N}\}\subset\{x_{n}: n\in \mathbb{N}\}$ such that $\{x\}\cup\{x_{n_{i}}: i\in \mathbb{N}\}\subset P\subset U$, then $\mathscr{P}_{x}$ is called a \emph{cs$^{\ast}$-network} at $x$ in $X$. The family $\mathcal{P}$ is called a \emph{$cs^{\ast}$-network}
of $X$ if each each $\mathscr{P}_{x}$ is a cs$^{\ast}$-network at $x$ for each $x\in X$.
\end{definition}

\begin{definition}\cite{G1984}
Let $\mathscr P$ be a family of subsets of a space $X$.

\smallskip
$\bullet$ The family $\mathscr P$ is called a {\it network} if for each $x\in X$ and any open neighborhood $U$ of $x$  there exists $P\in\mathscr {P}$ such that $x\in P\subset U$.

\smallskip
$\bullet$ The family $\mathscr P$ is called a {\it $k$-network} if for every compact subset $K$ of $X$ and an arbitrary open set $U$ containing $K$ in $X$ there is a finite subfamily $\mathscr {P}^{\prime}\subset\mathscr {P}$ such that $K\subset \bigcup\mathscr {P}^{\prime}\subset U$.

\smallskip
$\bullet$ The space $X$ is an \emph{$\aleph_{0}$-space} if $X$ has a countable $k$-network.

\smallskip
$\bullet$ The space $X$ is {\it cosmic} if it is a regular space with a countable network.
\end{definition}

 \smallskip
\section{A characterization of convergent sequences of $\mathcal{S}(X)$}
In this section, we give a characterization of convergent sequences in $\mathcal{S}(X)$, which plays an important role in this paper. First of all, we need some lemmas.

Clearly, each element $A$ of $\mathcal{S}(X)$ can be represented as the following form:
\begin{enumerate}
\item $A=\bigcup _{n=1}^{k} S_{n}$, where $S_{n}=\{s_{nj}\}_{j\in \mathbb{N}}\cup\{s_{n}\}$ and $s_{nj}\rightarrow s_{n}~(j\rightarrow\infty)$ for each $n\in\{1, ..., k\}$ such that $S_{n}\cap S_{m}=\emptyset$ for distinct $n, m\in\{1,...,k\}$.

\smallskip
\item For each $n\in\{1, ..., k\}$, if $S_{n}$ is a trivial convergent sequence then $s_{nj}=s_{n}$ for each $j\in\mathbb{N}$.
\end{enumerate}

Throughout this paper, we always say that $A$ has this kind of form above for any $A\in\mathcal{S}(X)$.

\begin{lemma}\label{11}
For any $A=\bigcup _{n=1}^{k} S_{n}\in \mathcal{S}(X)$, if $\widehat{\mathcal{U}}$ is an open neighborhood of $A$ in $\mathcal{S}(X)$, then there exist a positive integer $N\geq k$ and disjoint open sets $V_{1}, ..., V_{N}$ in $X$ such that the following conditions hold:
\begin{enumerate}
\item $s_{i}\in V_{i}$ for each $i\leq k$;

\smallskip
\item $V_{i}\subset X\setminus\bigcup_{j=1}^{k}V_{j}$ and $|V_{i}\cap (A\setminus\bigcup_{j=1}^{k}V_{j})|=1$ for each $k<i\leq N$;

\smallskip
\item $A\in \langle V_{1},...,V_{N}\rangle \subset \widehat{U}$.
\end{enumerate}
\end{lemma}

\begin{proof}
Since $\widehat{U}$ is an open neighborhood of $A$ in $\mathcal{S}(X)$, there exist open sets $U_{1},...,U_{t}$ in $X$ such that $A\in\langle U_{1}, ..., U_{t}\rangle \subset \widehat{U} $. Let $\Lambda_{n}=\{j: s_{n}\in U_{j}, j\leq t\}$ for any $n\in\{1,...,k\}$. Put $\Lambda=\{1, \cdots, t\}\setminus\bigcup_{n=1}^{k}\Lambda_{n}$. For each $i\in\Lambda$, take an arbitrary $b_{i}\in A\cap U_{i}$. Let $A_{t}=\{b_{i}: i\in\Lambda\}$. Moreover, it is easy to see that each $\Lambda_{n} \neq\emptyset.$ Since $X$ is Hausdorff and $s_{n}\in \bigcap_{j\in \Lambda_{n}}U_{j}$ for any $n\in\{1,...,k\}$, it follows that we can find disjoint open sets $V_{1},...,V_{k}$ in $X$ such that $s_{n}\in V_{n}\subset \bigcap _{j\in \Lambda_{n}}U_{j}$ and $V_{n}\cap A_{t}=\emptyset$. Set $K_{n}=(\{s_{nj}: j\in\mathbb{N}\}\cup\{s_{n}\})\cap V_{n}$ for any $n\in\{1,...,k\}$. Then the family $\{K_{n}: n\in\{1,...,k\}\}$ consists of disjoint compact and closed subsets of $X$ and $K_{n}\subset V_{n}$  for any $n\in\{1,...,k\}$.
Evidently, the set $X'=A\cap (X\setminus \bigcup _{n=1}^{k} V_{n})$ is finite. Therefore, we can set $X'\cup A_{t}=\{a_{k+1}, ..., a_{N}\}$. Then it follows that there exist mutually disjoint open sets $V_{k+1}, ..., V_{N}$ in $X$ satisfy the following conditions:

\smallskip
(a) $V_{n}\cap A=\{a_{n}\}$ for each $n\in \{k+1, \cdots, N\}$;

\smallskip
(b) $\bigcup_{n=k+1}^{N} V_{n}\subset \bigcup_{j=1}^{t}U_{j}$;

\smallskip
(c) for each $i\in\Lambda$, there exists $n\in\{k+1, \cdots, N\}$ such that $b_{i}\in V_{n}$ and $V_{n}\subset U_{i}$.

From \cite[Lemma 2.3.1]{M}, we can easily verify that $$A\in \langle V_{1}, ..., V_{N}\rangle \subset\langle U_{1}, ..., U_{t}\rangle  \subset\widehat{U}$$ and $V_{1}, ..., V_{N}$ satisfy the conditions (1)-(3) above.
\end{proof}

\begin{lemma}\label{16}
Let $X$ be a Hausdorff space and the sequence $\{A_{n}\}$ of $\mathcal{S}(X)$ converge to a point $A$ in $\mathcal{S}(X)$. Then $A\cup\bigcup_{n\in\mathbb{N}}A_{n}$ is a countable compact metrizable subspace of $X$.
\end{lemma}

\begin{proof}
Since $\{A_{n}\}$ converges to $A$, the set $B=A\cup\bigcup_{n\in\mathbb{N}}A_{n}$ is a compact subset in $X$ by \cite[Theorem 0.2]{IN}. Then $B$ is metrizable since a Hausdorff compact space with a countable network is metrizable \cite{G1984}, hence $B$ is a countable compact metrizable subspace of $X$.
\end{proof}

Now, we can prove the main result in this section.

\begin{theorem}\label{18}
Let $X$ be a space and $A\in \mathcal{S}(X)$. For any sequence $\{A_{n}\}$ in $\mathcal{S}(X)$, the following conditions (1)-(3) are equivalent:

\smallskip
(1) The sequence $\{A_{n}\}$ converges to $A$ in $\mathcal{S}(X)$.

\smallskip
(2) The sequence $\{A_{n}\}$ satisfies the following conditions:

\smallskip
(i) For any strictly increasing subsequence $\{n_{k}\}\subset\mathbb{N}$, we have $A=\bigcap_{n=1}^{\infty}\overline{\bigcup_{k\geq n}A_{n_{k}}}$;

\smallskip
(ii)For any neighborhood $U$ of $A$ in $X$, there is $N\in\mathbb{N}$ such that $A_{n}\subset U$ for any $n>N$.

\smallskip
(3) The sequence $\{A_{n}\}$ satisfies the following conditions:

\smallskip
(i$^{\prime}$) For any strictly increasing subsequence $\{n_{k}\}\subset\mathbb{N}$, we have $A=\bigcap_{n=1}^{\infty}\overline{\bigcup_{k\geq n}A_{n_{k}}}$;

\smallskip
(ii$^{\prime}$) For any sequential neighborhood $U$ of $A$ in $X$, there is $N\in\mathbb{N}$ such that $A_{n}\subset U$ for any $n>N$.
\end{theorem}

\begin{proof}
Obviously, (3) $\Rightarrow$ (2). It suffices to prove ~(2) $\Rightarrow$ (1) and ~(1) $\Rightarrow$ (3).

\smallskip
(2) $\Rightarrow$ (1). Assuming that the sequence $\{A_{n}\}$ satisfies the conditions (i) and (ii). Take an arbitrary open neighborhood $\hat{\mathcal{U}}$ of $A\in \mathcal{S}(X)$. Now we will prove that there exists $N\in\mathbb{N}$ such that $A_{n}\in \hat{\mathcal{U}}$ for any $n>N$. From Lemma~\ref{11}, we can find disjoint open sets $V_{1}, V_{2}, ~\cdots, V_{m}$ in $X$ such that $A\in \langle V_{1}, V_{2},\cdots, V_{m}\rangle\subset\hat{\mathcal{U}}$. For each $i\in\{1, \cdots, m\}$, it follows from (i) that there exists $N(i)\in\mathbb{N}$ such that $V_{i}\cap A_{n}\neq\emptyset$ for each $n>N(i)$. Put $N_{1}=\max\{N(i): i=1, 2, \cdots, m\}$. Then $A_{n}\cap V_{i}\neq\emptyset$ for any $i\in\{1, \cdots, m\}$ and $n>N_{1}$. Since $A\in \langle V_{1}, V_{2}, \cdots, V_{m}\rangle$, we can see that $A\subset\bigcup_{i=1}^{m}V_{i}$, then there exists $N_{2}\in \mathbb{N}$ by (ii) such that $A_{n}\subset\bigcup_{i=1}^{m}V_{i}$ for any $n>N_{2}$. Put $N=\max\{N_{1}, N_{2}\}$, then $A_{n}\in\langle V_{1}, V_{2}, \cdots, V_{m}\rangle\subset\hat{\mathcal{U}}$ for each $n>N$. Therefore,  the sequence $\{A_{n}\}$ converges to $A$ in $\mathcal{S}(X)$.

\smallskip
(1) $\Rightarrow$ (3). Assume that the sequence ~$\{A_{n}\}$ converges to ~$A$ in ~$\mathcal{S}(X)$. We will prove that sequence $\{A_{n}\}$ satisfies the conditions (i$^{\prime}$) and (ii$^{\prime}$).

\smallskip
(i$^{\prime}$) Suppose that there exists strictly increasing subsequence $\{n_{k}\}\subset\mathbb{N}$ such that $A\neq\bigcap_{n=1}^{\infty}\overline{\bigcup_{k\geq n}A_{n_{k}}}=B$. Obviously, the subsequence $\{A_{n_{k}}\}$ also converges to $A$. In order to obtain a contradiction, we divide the proof into the following two cases:

\smallskip
{\bf case 1:} $A\setminus B\neq\emptyset$.

\smallskip
Take any $a\in A\setminus B$. Then $a\not\in B$, thus there exists $N\in\mathbb{N}$ such that $a\not\in \overline{\bigcup_{k>N}A_{n_{k}}}$. Then there is an open neighborhood of $a$ in $X$ such that $V\cap \overline{\bigcup_{k>N}A_{n_{k}}}=\emptyset$. Let $\widehat{U}=\langle X, V\rangle$, hence $\widehat{U}$ is an open neighborhood of $A$ in $\mathcal{S}(X)$, but $A_{n_{k}}\not\in\widehat{U}$ for any $k>N$, which is a contradiction with $\{A_{n_{k}}\}$ converging to $A$ in $\mathcal{S}(X)$.

\smallskip
{\bf case 2:} $B\setminus A\neq\emptyset$.

\smallskip
Take any $b\in B\setminus A$. Then $b\not\in A$, hence from the compactness of $A$ in $X$ it follows that there exists an open set $V$ of $b$ in $X$ such that $\overline{V}\cap A=\emptyset$. Put $\widehat{U}=\langle X\setminus \overline{V}\rangle$. Then $\widehat{U}$ is an open neighborhood of $A$ in $\mathcal{S}(X)$. However,  since $b\in B\setminus A$, there exists a subsequence $\{A_{n_{k_{l}}}\}$ of $\{A_{n_{k}}\}$ such that $V\cap A_{n_{k_{l}}}\neq\emptyset$ for each $l\in\mathbb{N}$, hence $A_{n_{k_{l}}}\not\subset X\setminus \overline{V}$, that is, $A_{n_{k_{l}}}\not\in \widehat{U}$, which is a contradiction.

\smallskip
(ii$^{\prime}$) Suppose not, then we can find a sequential neighborhood $U$ of $A$ in $X$ such that for any $k\in\mathbb{N}$, there is $n_{k}>k$ satisfying $A_{n_{k}}\setminus U\neq\emptyset$. Without loss of generality, we may assumed that $\{n_{k}\}_{k\in \mathbb{N}}$ is strictly increasing. For any $k\in\mathbb{N}$, take $a_{k}\in A_{n_{k}}\setminus U$. Since ~$A\subset U$, we see that $a_{k}\not\in A$. From ~(i$^{\prime}$), it easily see that $\{a_{k}: k\in\mathbb{N}\}$ is an infinite set. By Lemma~\ref{16}, $\{a_{k}: k\in\mathbb{N}\}$ must have a convergent subsequence and let $\{a_{k_{l}}: l\in\mathbb{N}\}$ converge to $a$. From (i$^{\prime}$), we know that $a\in A$, then there is $N\in\mathbb{N}$ such that $a_{k_{l}}\in U$ for any $l>N$, which is a contradiction. Consequently, there exists $N\in \mathbb{N}$ such that $A_{n}\subset U$ for any $n>N$.
\end{proof}

Finally, we study the hyperspace $\mathcal{S}(X)$ with a $cs$-network.

\begin{lemma}\label{12}
Let the sequence $\{A_{n}\}_{n\in N}$ in $\mathcal{S}(X)$ converge to a point $A=\bigcup _{n=1}^{k} S_{n}$. If $V_{1}, ..., V_{N}$ are mutually disjoint open sets in $X$ which satisfy the conditions in Lemma~\ref{11}, then sequence $\{A_{n}\cap V_{m}\}_{n\in \mathbb{N}}$ converges to the set $A\cap V_{m}$ in $X$ for each $m\in \{1, ..., N\}$.
\end{lemma}

\begin{proof}
Fix $m\in \{1,...,N\}$, let $U_{m}$ be a neighborhood of $A\cap V_{m}$ in $X$, and let $O_{m}=V_{m}\cap U_{m}$. Then $O_{m}$ is a neighborhood of $A\cap V_{m}$ in $X$, hence it follows that $$A\in \langle V_{1},..., V_{m-1}, O_{m}, V_{m+1},..., V_{N}\rangle.$$ Since $\{A_{n}\}_{n\in \mathbb{N}}$ converges to the point $A$, there exists $M\in \mathbb{N}$ such that $$\{A_{n}: n>M\}\subset \langle V_{1}, ..., V_{m-1}, O_{m}, V_{m+1},..., V_{N}\rangle\cap \langle V_{1},..., V_{N}\rangle.$$ Because $V_{1},...,V_{N}$ are mutually disjoint open sets and $O_{m}\subset V_{m}$, then $A_{n}\cap V_{m} \subset O_{m}\subset U_{m}$ for any $n>M$. Therefore, $\{A_{n}\cap V_{m}\}_{n\in \mathbb{N}}$ converges to the set $A\cap V_{m}$ in $X$.
\end{proof}

\smallskip
\begin{theorem}
If $\mathcal{N}$ is a $CS$-network of a space $X$, then $$\mathcal{B}=\{\langle \mathcal{P}\rangle: \mathcal{P}\ \mbox{is a finite subset of}\ \mathcal{N}\}$$ is a $cs$-network of $\mathcal{S}(X)$.
\end{theorem}

\begin{proof}
Suppose a sequence $\{A_{n}\}_{n\in \mathbb{N}}\subset \mathcal{S}(X)$ converges to a point $A$ in $\mathcal{S}(X)$. Put $A=\bigcup _{n=1}^{k} S_{n}$, and let $\widehat{\mathcal{U}}$ be a neighborhood of $A$ in $\mathcal{S}(X)$. By Lemma~\ref{11}, there exists mutually disjoint open sets $V_{1}, ..., V_{N}$ in $X$ such that $A\in\langle V_{1}, ..., V_{N}\rangle \subset \widehat{\mathcal{U}}$. Since $\{A_{n}\}_{n\in N}$ converges to $A$, we can find $N_{0}\in\mathbb{N}$ such that $A_{n}\in\langle V_{1},..., V_{N}\rangle$ for any $n>N_{0}$. It follows from Lemma~\ref{12} that $\{A_{n}\cap V_{m}\}_{n\in \mathbb{N}}$ converges to $A\cap V_{m}$ in $X$ for any $m\in\{1,...,N\}$. Since $\mathcal{N}$ is a $CS$-network for $X$, we conclude that $$(A\cap V_{m})\cup \bigcup\{A_{n}\cap V_{m}: n>N_{m}\}\subset P_{m}\subset V_{m}$$ for some $N_{m}\in \mathbb{N}$ and $P_{m}\in \mathcal{N}$. Set $M=\max\{N_{m}: m\leq N\}$, then $$\{A\}\cup \{A_{n}:n>M\}\subset \langle P_{1},...,P_{N}\rangle\subset \langle V_{1},...,V_{n}\rangle\subset \widehat{\mathcal{U}}.$$ Consequently, $\mathcal{B}$ is a $cs$-network of $\mathcal{S}(X)$.
\end{proof}

 \smallskip
\section{Some cardinal invariants on $\mathcal{S}(X)$}
In this section, we mainly discuss some cardinal invariants on $\mathcal{S}(X)$, such as, the pseudocharacter, the character, the $sn$-character, the $sn\omega$-weight, etc. First, we recall some concepts.

\begin{definition}
Let $X$ be a space.
\begin{enumerate}
\item $\pi (X)=\omega+min\{|\mathcal{B}|:\mathcal{B}\ \mbox{is}\ \pi\mbox{-base}\ \mbox{of}\ X\}$.

\smallskip
\item $d(X)=\omega+min\{|S|:S\ \mbox{is a dense subset of}\ X\}$.

\smallskip
\item $\chi(x, X)=\omega+min\{|\mathcal{U}|:\mathcal{U}\ \mbox{is a neighborhood base of}\ x\ \mbox{in}\ X\}$.

\smallskip
\item $\psi(x, X)=\omega+min\{|\mathcal{G}|:\mathcal{G}\ \mbox{is a family of open sets in}\ X\ \mbox{such that}\ \cap\mathcal{G}=\{x\}\}$.

\smallskip
\item $\chi(X)=sup\{\chi (x, X):x\in X\}$; in particular, we say that $X$ is {\it first-countable} if $\chi(X)\leq\aleph_{0}$.

\smallskip
\item $\psi(X)=sup\{\psi(x, X):x\in X\}$; in particular, we say that $X$ is {\it of countable pseudocharacter} if $\psi(X)\leq\aleph_{0}$.

\smallskip
\item $\omega(X)=\omega+min\{|\mathcal{B}|:\mathcal{B}\ \mbox{is a base of}\ X\}$.
\end{enumerate}
\end{definition}

Clearly, we have the following proposition.

\begin{proposition}
Let $X$ be a space. Then $d(X)=d(\mathcal{S}(X))$.
\end{proposition}

The following three theorems show that the relations of the character, the pseudocharacter and the $\pi$-character in a space $X$ and its hyperspace $\mathcal{S}(X)$. From \cite[Theorem 1.2]{IN}, we can find a first-countable space $X$ and a space $Y$ with countable pseudocharacter such that $\mathcal{K}(X)$ is not first-countable and $\mathcal{K}(Y)$ is not of countable pseudocharacter respectively. However, the following Theorems~\ref{mingti 10} and \ref{mingti 7} show that the situations are different for the class of hyperspaces of finite unions of convergent sequences.

\begin{theorem}\label{mingti 10}
$\psi(\mathcal{S}(X))=\psi(X)$.
\end{theorem}

\begin{proof}
Clearly, we have $\psi(\mathcal{S}(X))\geq\psi(X)$. It suffices to prove that $\psi(\mathcal{S}(X))\leq\psi(X)$. Let $\kappa=\psi(X)$. Take an arbitrary $K\in \mathcal{S}(X)$. Denote $K$ by $\{a_{n}: n\in\mathbb{N}\}$. For each $n\in \mathbb{N}$, there exist a family $\mathscr{B}_{n}$ of open sets in $X$ and a countable subfamily $\mathscr{B}_{n}^{\prime}\subset \mathscr{B}_{n}$ such that $|\mathscr{B}_{n}|\leq\kappa$, $\{a_{n}\}=\bigcap\mathscr{B}_{n}$ and  $\mathscr{B}_{n}^{\prime}|_{K}$ is a local base at $a_{n}$ in $K$. Put $\mathscr{B}=\bigcup_{n\in\mathbb{N}}\mathscr{B}_{n}$ and $$\mathscr{P}=\{\langle\mathscr{B}^{\prime}\rangle: \mathscr{B}^{\prime}\in\mathscr{B}^{<\omega},\ K\subset\bigcup\mathscr{B}^{\prime}\ \mbox{and}\ K\cap U\neq\emptyset\ \mbox{for any}\ U\in\mathscr{B}^{\prime}\},$$where $\langle\mathscr{B}^{\prime}\rangle=\langle U_{1}, \cdots, U_{m}\rangle$ if $\mathscr{B}^{\prime}=\{U_{i}: i=1, \cdots, m\}$. Clearly, $|\mathscr{P}|\leq\kappa$. We claim that $\{K\}=\bigcap\mathscr{P}$. Indeed, it is obvious that $K\in\bigcap\mathscr{P}$. Take any $L\in \mathcal{S}(X)\setminus\{K\}$. If $L\setminus K\neq\emptyset$, then take any $x\in L\setminus K$. Hence, for each $n\in \mathbb{N}$, there exists $U_{n}\in\mathscr{B}_{n}$ such that $x\not\in U_{n}$. Obviously, $\{U_{n}: n\in\mathbb{N}\}$ is an open cover of $K$, hence there exists a finite subcover $\{U_{n_{i}}: i\leq N\}$ for some $N\in\mathbb{N}$ such that $K\cap U_{i}\neq\emptyset$ for each $i\leq N$. Then $\langle U_{n_{1}}, \ldots, U_{n_{N}}\rangle\in \mathscr{P}$, but  $L\not\in\langle U_{n_{1}}, \ldots, U_{n_{N}}\rangle$, thus $L\not\in\bigcap\mathscr{P}$. If $L\subset K$ and $K\setminus L\neq\emptyset$, then take any $x\in K\setminus L$. Let $x=a_{n}$. From the compactness of $L$ and $K$, it easily verify that there exists $B\in \mathscr{B}_{n}^{\prime}$ such that $x\in B$ and $B\cap L=\emptyset$. Then, for any $\mathscr{B}^{\prime}\in\mathscr{B}^{<\omega}$, we have $L\not\in\langle\mathscr{B}^{\prime}\cup\{B\}\rangle$, hence $L\not\in\bigcap\mathscr{P}$.

Therefore, $\bigcap\mathscr{P}=\{K\}$.
\end{proof}

\begin{theorem}\label{mingti 7}
$\chi(\mathcal{S}(X))=\chi(X)$.
\end{theorem}

\begin{proof}
Clearly, we have $\chi(X)\leq\chi(\mathcal{S}(X))$. It only need to prove that $\chi(\mathcal{S}(X))\leq\chi(X)$. Take an arbitrary $K\in \mathcal{S}(X)$. Denote $K$ by $\{a_{n}: n\in\mathbb{N}\}$. For each $n\in \mathbb{N}$, there exists a family $\mathscr{B}_{n}$ of open sets in $X$ such that $|\mathscr{B}_{n}|\leq\kappa$ and $\mathscr{B}_{n}$ is a base at point $a_{n}$ in $X$. Put $\mathscr{B}=\bigcup_{n\in\mathbb{N}}\mathscr{B}_{n}$ and $$\mathscr{P}=\{\langle\mathscr{B}^{\prime}\rangle: \mathscr{B}^{\prime}\in\mathscr{B}^{<\omega},\ K\subset\bigcup\mathscr{B}^{\prime}\ \mbox{and}\ K\cap U\neq\emptyset\ \mbox{for any}\ U\in\mathscr{B}^{\prime}\},$$ where $\langle\mathscr{B}^{\prime}\rangle=\langle U_{1}, \cdots, U_{m}\rangle$ if $\mathscr{B}^{\prime}=\{U_{i}: i=1, \cdots, m\}$. Clearly, $|\mathscr{P}|\leq\kappa$. We claim that $\mathscr{P}$ is a base at point $K$ in $\mathcal{S}(X)$. Indeed, for any open neighborhood $\widehat{U}$ of $K$ in $\mathcal{S}(X)$. From Lemma~\ref{11}, there exist disjoint open sets $V_{1}, \cdots, V_{N}$ in $X$ such that $K\in\langle V_{1}, \cdots, V_{N}\rangle\subset \widehat{U}$. For each $i\leq N$ and any $x\in V_{i}\cap K$, there exists an open neighborhood $U_{x, i}\in\mathscr{B}_{n}$ such that $U_{x, i}\subset V_{i}$. Put $\mathscr{O}=\{U_{x, i}: x\in V_{i}\cap K, i\leq N\}$. Then $\mathscr{O}$ is an open cover of $K$, therefore there exists a finite subcover $\mathscr{O}^{\prime}$. Then it easily verify that $\langle\mathscr{O}^{\prime}\rangle\subset \langle V_{1}, \cdots, V_{N}\rangle\subset \widehat{U}$. Therefore, $\mathscr{P}$ is a base at point $K$ in $\mathcal{S}(X)$. By the arbitrary choice of $K$, we have that $\chi(\mathcal{S}(X))\leq\chi(X)$.
\end{proof}

By \cite[Theorem 4]{TM}, we have the following theorem.

\begin{theorem}\label{mingti 9}
(1) $\pi(X)=\pi(\mathcal{S}(X))$;  (2) $\omega(X)=\omega(\mathcal{S}(X))$.
\end{theorem}

The following corollaries are straightforward from the above three theorems respectively.

\begin{corollary}
A space $X$ is first-countable if and only if the hyperspace $\mathcal{S}(X)$ is first countable.
\end{corollary}

\begin{corollary}
A space $X$ is of countable pseudocharacter if and only if the hyperspace $\mathcal{S}(X)$ is of countable pseudocharacter.
\end{corollary}

\begin{corollary}
A space $X$ has a countable $\pi$-character if and only if the hyperspace $\mathcal{S}(X)$ has a countable $\pi$-character.
\end{corollary}

Next we will consider some cardinal invariants of generalized metric spaces of the hyperspace $\mathcal{S}(X)$.

The smallest size $|\mathcal{N}|$ of an $sn$-network (resp. $so$-network, $cs$-network, $cs^{\ast}$-network) $\mathcal{N}$ at
a point $x\in X$ is called the {\it $sn$-character} (resp. {\it $so$-character, $cs$-character, $cs^{\ast}$-character}) of
$X$ at the point $x$ and is denoted by $sn_{\chi}(X, x)$ (resp. $so_{\chi}(X, x)$, $cs_{\chi}(X, x)$, $cs^{\ast}_{\chi}(X, x)$). The
cardinals $sn_{\chi}(X)=\sup_{x\in X}sn_{\chi}(X, x)$, $so_{\chi}(X)=\sup_{x\in X}so_{\chi}(X, x)$, $cs_{\chi}(X)=\sup_{x\in X}cs_{\chi}(X, x)$ and $cs^{\ast}_{\chi}(X)=\sup_{x\in X}cs^{\ast}_{\chi}(X, x)$ are called the {\it $sn$-character, $so$-character, $cs$-character and $cs^{\ast}$-character} of the topological space $X$, respectively. For the empty topological space $X=\emptyset$, we put $sn_{\chi}(X)=so_{\chi}(X)=cs_{\chi}(X)=cs^{\ast}_{\chi}(X)=1$. Moreover, we define the following some cardinal invariant on $X$.

\smallskip
$sn\omega(X)=\aleph_{0}+\min\{|\mathcal{N}|: \mathcal{N}\ \mbox{is an sn-network of}\ X\}$;

\smallskip
$so\omega(X)=\aleph_{0}+\min\{|\mathcal{N}|: \mathcal{N}\ \mbox{is an so-network of}\ X\}$;

\smallskip
$cs\omega(X)=\aleph_{0}+\min\{|\mathcal{N}|: \mathcal{N}\ \mbox{is a cs-network of}\ X\}$;

\smallskip
$cs^{\ast}\omega(X)=\aleph_{0}+\min\{|\mathcal{N}|: \mathcal{N}\ \mbox{is a}\ cs^{\ast}\mbox{-network of}\ X\}$.

\smallskip
The following question is interesting.

\begin{question}\label{q1}
Let $X$ be a topological space. Which the following equalities hold?

\smallskip
(1) $sn_{\chi}(X)=sn_{\chi}(\mathcal{S}(X))$;

\smallskip
(2) $so_{\chi}(X)=so_{\chi}(\mathcal{S}(X))$;

\smallskip
(3) $cs_{\chi}(X)=cs_{\chi}(\mathcal{S}(X))$;

\smallskip
(4) $cs^{\ast}_{\chi}(X)=cs^{\ast}_{\chi}(\mathcal{S}(X))$;

\smallskip
(5) $sn\omega(X)=sn\omega(\mathcal{S}(X))$;

\smallskip
(6) $so\omega(X)=so\omega(\mathcal{S}(X))$;

\smallskip
(7) $cs\omega(X)=cs\omega(\mathcal{S}(X))$;

\smallskip
(8) $cs^{\ast}\omega(X)=cs^{\ast}\omega(\mathcal{S}(X))$.
\end{question}

Now we will give some answers to Question~\ref{q1}. First, we give negative answers to (1) and (5) of Question~\ref{q1}.

\begin{example}
There exists an $snf$-countable space $X$ such that $\mathcal{S}(X)$ is not $snf$-countable; in particular, $sn\omega(\mathcal{S}(X))>sn\omega(X)$.
\end{example}

\begin{proof}
Let $X$ be the $S_{2}$-space. Then $X$ is $snf$-countable. It follows from Proposition~\ref{ming ti 5} in Section 6 that $\mathcal{S}(X)$ contains a copy of $S_{\omega}$. But $S_{\omega}$ is not $snf$-countable, hence $\mathcal{S}(X)$ is not $snf$-countable. Moreover, it is obvious that $sn\omega(\mathcal{S}(X))>sn\omega(X)$.
\end{proof}

However, the equalities (1) and (5) of Question~\ref{q1} in $\mathcal{F}(X)$ hold.

\begin{theorem}\label{tt}
For any space $X$, we have $sn_{\chi}(X)=sn_{\chi}(\mathcal{F}(X))$.
\end{theorem}

\begin{proof}
Clearly, $sn_{\chi}(X)\leq sn_{\chi}(\mathcal{F}(X))$. It suffices to prove that $sn_{\chi}(\mathcal{F}(X))\leq sn_{\chi}(X)$. Let $sn_{\chi}(X)=\kappa$, and let $\mathcal{P}=\bigcup_{x\in X}\mathcal{P}_{x}$ be a sn-network of $X$, where each $\mathcal{P}_{x}$ is an sn-network of point $x$ and $|\mathcal{P}_{x}|\leq\kappa$. Clearly, it only need to prove $sn_{\chi}(\mathcal{F}(X), A)\leq\kappa$ for each $A\in\mathcal{F}(X)$. Take an arbitrary $A\in\mathcal{F}(X)$. Let $A=\{x_{1}, \cdots, x_{n}\}$, and put
$$\widehat{\mathcal{P}}_{A}=\{\langle P_{x_{1}}, \cdots, P_{x_{n}}\rangle\cap \mathcal{F}(X): P_{x_{i}}\in \mathcal{P}_{x_{i}}, P_{x_{i}}\cap P_{x_{j}}=\emptyset,~i\neq j,~i,j\leq n\}.$$ Then $\widehat{\mathcal{P}}_{A}$ is an sn-network of point $A$ in $\mathcal{F}(X)$ and $|\widehat{\mathcal{P}}_{A}|\leq\kappa$. In fact, it is obvious that $|\widehat{\mathcal{P}}_{A}|\leq\kappa$ and $\widehat{\mathcal{P}}_{A}$ is a network at point $A$ in $\mathcal{F}(X)$. It suffices to prove that each element of $\widehat{\mathcal{P}}_{A}$ is a sequential neighborhood of the point $A$ in $\mathcal{F}(X)$. Take any $\widehat{U}\in \widehat{\mathcal{P}}_{A}$. Then $\widehat{U}=\langle P_{x_{1}}, \cdots, P_{x_{n}}\rangle\cap \mathcal{F}(X)$, where each $P_{x_{i}}\in \mathcal{P}_{x_{i}}$ and $P_{x_{i}}\cap P_{x_{j}}=\emptyset$ if $i\neq j$. Let the sequence $\{A_{n}\}\subset\mathcal{F}(X)$ converge to $A$ in $\mathcal{F}(X)$. Since $U=\bigcup_{i=1}^{n}P_{x_{i}}$ is a sequential neighborhood of $A$ in $X$, it follows from Theorem \ref{18} that there is $N\in\mathbb{N}$ such that $A_{n}\subset U$ for any $n>N$. Without loss of generality, we may assume that $A_{n}\subset U$ for each $n\in\mathbb{N}$. Next we prove that there is $N_{1}\in\mathbb{N}$ such that $A_{n}\in \widehat{U}$ for each $n>N_{1}$. If not, there are subsequences $\{A_{n_{k}}\}$ and $i_{0}\leq n$ such that $A_{n_{k}}\not\in\widehat{U}$ and $A_{n_{k}}\cap P_{x_{i_{0}}}=\emptyset$ for each $k\in\mathbb{N}$. Obviously, $x_{i_{0}}\not\in A_{n_{k}}$, and it follows from Lemma \ref{16} that $K=A\cup \bigcup_{k\in\mathbb{N}}A_{n_{k}}$ is a countable compact metrizable space. However, from Theorem \ref{18}, we see that $x_{i_{0}}$ is the limit point in $K$, hence there is a non-trivial convergence sequence in $K\setminus A$ converges to $x_{i_{0}}$, which contradicts with $A_{n_{k}}\cap P_{x_{i_{0}}}=\emptyset$. Therefore, each element of $\widehat{\mathcal{P}}_{A}$ is a sequential neighborhood of the point $A$ in $\mathcal{F}(X)$. Then $sn_{\chi}(\mathcal{F}(X), A)\leq\kappa$.
\end{proof}

\begin{corollary}
A space $X$ is $snf$-countable if and only if $\mathcal{F}(X)$ is $snf$-countable.
\end{corollary}

\begin{proposition}
If $\mathcal{N}$ is a $sn$-network of a space $X$, then $$\mathcal{B}=\{\langle B_{1}, ..., B_{n}\rangle \cap \mathcal{F}(X):B_{i}\in \mathcal{N}, i\leq n, n\in\mathbb{N}\}$$ is a $sn$-network of $\mathcal{F}(X)$.
\end{proposition}

\begin{proof}
Since $\mathcal{N}$ is a $sn$-network of a space $X$, let $\mathcal{N}=\bigcup_{x\in X}\mathcal{N}_{x}$, where $\mathcal{N}_{x}$ is an $sn$-network at point $x$ in $X$. Then it easily see that $\mathcal{B}$ is a $sn$-network of $\mathcal{F}(X)$ from the proof of Theorem~\ref{tt}, (ii$^{\prime}$) of (3) in Theorem~\ref{18} and the following fact:

\smallskip
{\bf Fact:} For any $A=\{x_{1}, \cdots, x_{n}\}\in\mathcal{F}(X)$ and any $P_{i}\in\mathcal{N}_{x_{i}}$ for each $i\leq n$, the set $\bigcup\{P_{i}: i\leq n\}$ is a sequential neighborhood of $A$ in $X$.
\end{proof}

\begin{theorem}
For any space $X$, we have $sn\omega(X)=sn\omega(\mathcal{F}(X))$.
\end{theorem}

\begin{theorem}\label{tt1}\label{tt7}
For any space $X$, we have $so_{\chi}(X)=so_{\chi}(\mathcal{S}(X))$.
\end{theorem}

\begin{proof}
Clearly, $so_{\chi}(X)\leq so_{\chi}(\mathcal{S}(X))$. It suffices to prove that $so_{\chi}(\mathcal{S}(X))\leq so_{\chi}(X)$. Let $so_{\chi}(X)=\kappa$, and $\mathcal{P}=\bigcup_{x\in X}\mathcal{P}_{x}$ be an so-network of $X$, where each $\mathcal{P}_{x}$ is an so-network of point $x$ such that $|\mathcal{P}_{x}|\leq\kappa$. Clearly, it needs only to prove that $so_{\chi}(\mathcal{S}(X), A)\leq\kappa$ for each $A\in \mathcal{S}(X)$. Take an arbitrary $A\in \mathcal{S}(X)$. Set $A=\bigcup _{n=1}^{k} S_{n}$, and define $\widehat{\mathcal{P}}_{A}$ as the family of sets $\widehat{P}=\langle P_{s_{1}}, \cdots, P_{s_{k}}, P_{x_{k+1}}, \cdots, P_{x_{n}}\rangle$ satisfies the following conditions:

\begin{enumerate}
\item $P_{s_{i}}\in \mathcal{P}_{s_{i}}, ~P_{x_{i}}\in \mathcal{P}_{x_{i}}, ~x_{i}\in A\setminus\bigcup_{i=1}^{k}P_{s_{i}}$;

\smallskip
\item $s_{1}, \cdots, s_{k}$ are the limit points of $A$ and $|P_{x_{i}}\cap (A\setminus\bigcup_{j=1}^{k}P_{s_{j}})|=1$ for any $k<i\leq n$;

\smallskip
\item Any two elements of $\{P_{s_{1}}, \cdots, P_{s_{k}}, P_{x_{k+1}}, \cdots, P_{x_{n}}\}$ are disjoint;

\smallskip
\item If $A$ is a finite set, then $\widehat{P}=\langle P_{x_{k+1}}, \cdots, P_{x_{n}}\rangle$, where $A=\{x_{k+1}, \cdots, x_{n}\}$.
\end{enumerate}

We claim that $\widehat{\mathcal{P}}_{A}$ is an so-network of point $A$ in $\mathcal{S}(X)$ and $|\widehat{\mathcal{P}}_{A}|\leq\kappa$. Indeed, it is obvious that $|\widehat{\mathcal{P}}_{A}|\leq\kappa$. Then we first prove that $\widehat{\mathcal{P}}_{A}$ is a network of the point $A$ in $\mathcal{S}(X)$. Let $\widehat{U}$ be any open neighborhood of $A$ in $\mathcal{S}(X)$. Then, from Lemma~\ref{11}, there are disjoint open sets $V_{1}, \cdots, V_{N}$ in $X$ for some positive integers $N\geq k$ such that $V_{1}, \cdots, V_{N}$ satisfy the conditions (1)-(3) of Lemma~\ref{11}. For any $1\leq i\leq k$, there is $P_{s_{i}}\in \mathcal{P}_{s_{i}}$ such that $s_{i}\in P_{s_{i}}\subset V_{i}$. Let $A\setminus\bigcup_{j=1}^{k}P_{s_{j}}=\{x_{k+1}, \cdots, x_{n}\}$. Hence for any $k+1\leq i\leq n,$ there exists $V_{x_{i}}\in\{V_{1}, \cdots,V_{N}\}$ such that $x_{i}\in V_{x_{i}}$, thus $x_{i}\in P_{x_{i}}\subset V_{x_{i}}$ for some $P_{x_{i}}\in \mathcal{P}_{x_{i}}$. Then $$A\in \langle P_{s_{1}}, \cdots, P_{s_{k}}, P_{x_{k+1}}, \cdots, P_{x_{n}}\rangle\subset \langle V_{1}, \cdots, V_{N}\rangle\subset \widehat{U}.$$ Now it only needs to prove that each element of $\widehat{\mathcal{P}}_{A}$ is a sequentially open set in $\mathcal{S}(X)$. Take an arbitrary $\widehat{P}\in \widehat{\mathcal{P}}_{A}$, and let $\widehat{P}=\langle P_{s_{1}}, \cdots, P_{s_{k}}, P_{x_{k+1}}, \cdots, P_{x_{n}}\rangle$. Let $B\in \widehat{P}$, and let the sequence $\{B_{n}\}$ converge to $B$ in $\mathcal{S}(X)$. Since $\mathcal{P}=\bigcup_{x\in X}\mathcal{P}_{x}$ is an so-network of $X$, it easily see that $U=(\bigcup_{i=1}^{k}P_{s_{i}})\cup(\bigcup_{i=k+1}^{n}P_{x_{i}}) $ is a sequential neighborhood of the set $B$ in $X$. From Theorem~\ref{18}, we can see that there is $M\in\mathbb{N}$ such that $B_{n}\subset U$ for any $n>M$. Without loss of generality, we may assume that each $B_{n}$ is contained in $U$. We conclude that there is $N_{1}\in\mathbb{N}$ such that $B_{n}\in \widehat{P}$ for any $n>N_{1}$. If not, there are subsequences $\{B_{n_{k}}\}$ of $\{B_{n}\}$ and $x\in \{s_{1}, \cdots,s_{k}, x_{k+1}, \cdots, x_{n}\}$ such that $B_{n_{k}}\not\in\widehat{P}$ and $B_{n_{k}}\cap P_{{x}}=\emptyset$ for any $k\in\mathbb{N}$. Therefore, $P_{{x}}\cap\bigcup_{k\in\mathbb{N}}B_{n_{k}}=\emptyset$. Since $B\in \widehat{P}$, we have $B\cap P_{{x}}\neq\emptyset$. Take any fixed point $b\in B\cap P_{{x}}$, then $P_{x}$ is a sequential neighborhood of the point $b$ in $X$. However, from Lemma~\ref{16} and Theorem~\ref{18} it follows that $b$ is the limit point of a sequence in $\bigcup_{k\in\mathbb{N}}B_{n_{k}}$, which is a contradiction with  $P_{{x}}\cap(\bigcup_{k\in\mathbb{N}}B_{n_{k}})=\emptyset$. Therefore, $so_{\chi}(\mathcal{S}(X), A)\leq\kappa$.
\end{proof}

\begin{corollary}
A space $X$ is $sof$-countable if and only if $\mathcal{S}(X)$ is $sof$-countable.
\end{corollary}

From the proof of Theorem~\ref{tt1}, we have the following Theorem.

\begin{theorem}\label{tt6}
For any space $X$, we have $so\omega(X)=so\omega(\mathcal{S}(X))$.
\end{theorem}

The above Theorems~\ref{tt7} and ~\ref{tt6} give affirmative answers to (2) and (6) in Question~\ref{q1} respectively.

\begin{proposition}\label{pp}
If $\mathcal{N}$ is a $cs$-network (resp. $cs^{\ast}$-network) of a space $X$, then $$\mathcal{B}=\{\langle B_{1}, ..., B_{n}\rangle \cap \mathcal{F}(X):B_{i}\in \mathcal{N}, i\leq n, n\in\mathbb{N}\}$$ is a $cs$-network (resp. $cs^{\ast}$-network) of $\mathcal{F}(X)$.
\end{proposition}

\begin{proof}
Suppose a sequence $\{A_{n}\}_{n\in \mathbb{N}}\subset \mathcal{F}(X)$ converges to a point $A=\{x_{1},\cdots,x_{m}\}$ in $\mathcal{F}(X)$. Put $K=A\cup\bigcup\{A_{n}: n\in\mathbb{N}\}$. From Lemma~\ref{16} and (ii) of (2) in Lemma~\ref{18}, $K$ is the union of finitely many convergent sequences in $X$ and the set of the limit points of $K$ is contained in $A$. Let $\widehat{\mathcal{U}}$ be a neighborhood of $A$ in $\mathcal{F}(X)$. Then there exist disjoint open sets $V_{1}, \cdots, V_{m}$ such that $A\in\langle V_{1}, \cdots, V_{m}\rangle\cap\mathcal{F}(X)\subset\widehat{\mathcal{U}}$ and $x_{i}\in V_{i}$ for each $i\leq m$. By Lemma~\ref{12}, $\{A_{n}\cap V_{i}\}_{n\in \mathbb{N}}$ converges to $A\cap V_{i}=\{x_{i}\}$ in $X$ for each $i\leq m$.

\smallskip
(1) If $\mathcal{N}$ is a $cs$-network of $X$, then there exist $B_{i}\in\mathcal{N}$ and $N_{i}\in\mathbb{N}$ such that $$\{x_{i}\}\cup\bigcup\{A_{n}\cap V_{i}: n\geq N_{i}\}\subset B_{i}\subset V_{i}.$$ Put $N=\max\{N_{i}: i\leq m\}$. Then $\langle B_{1}, \cdots, B_{m}\rangle\cap\mathcal{F}(X)\in\mathcal{B}$ and $$\{A_{n}: n>N\}\cup\{A\}\subset\langle B_{1}, \cdots, B_{m}\rangle\cap\mathcal{F}(X)\subset\langle V_{1}, \cdots, V_{m}\rangle\cap\mathcal{F}(X)\subset\widehat{\mathcal{U}}.$$

\smallskip
(2) If $\mathcal{N}$ is a $cs^{\ast}$-network of $X$, by induction on $m$ then there exist $B_{1}, \cdots, B_{m}\in\mathcal{N}$ and a subsequence $\{n_{k}\}_{k\in \mathbb{N}}$ of $\mathbb{N}$ such that $\{x_{i}\}\cup\bigcup\{A_{n_{k}}\cap V_{i}: k\in\mathbb{N}\}\subset B_{i}\subset V_{i}$. Then $\langle B_{1}, \cdots, B_{m}\rangle\cap\mathcal{F}(X)\in\mathcal{B}$ and $$\{A_{n_{k}}: k\in\mathbb{N}\}\cup\{A\}\subset\langle B_{1}, \cdots, B_{m}\rangle\cap\mathcal{F}(X)\subset\langle V_{1}, \cdots, V_{m}\rangle\cap\mathcal{F}(X)\subset\widehat{\mathcal{U}}.$$
\end{proof}

By Proposition~\ref{pp}, we have the following result.

\begin{theorem}\label{tt3}
For any space $X$, we have $$cs\omega(X)=cs\omega(\mathcal{F}(X))\ \mbox{and}\ cs^{\ast}\omega(X)=cs^{\ast}\omega(\mathcal{F}(X)).$$
\end{theorem}

\begin{theorem}\label{tt2}
For any space $X$, we have $$cs_{\chi}(X)=cs_{\chi}(\mathcal{F}(X))\ \mbox{and}\ cs_{\chi}^{\ast}(X)=cs_{\chi}^{\ast}(\mathcal{F}(X)).$$
\end{theorem}

\begin{proof}
We only consider the case of $cs_{\chi}(X)$, and the proof of $cs_{\chi}^{\ast}(X)$ is similar. Clearly, $cs_{\chi}(X)\leq cs_{\chi}(\mathcal{F}(X))$. It suffices to prove that $cs_{\chi}(\mathcal{F}(X))\leq cs_{\chi}(X)$. Let $cs_{\chi}(X)=\kappa$, and let $\mathcal{P}=\bigcup_{x\in X}\mathcal{P}_{x}$ be a $cs$-network of $X$, where each $\mathcal{P}_{x}$ is a $cs$-network at point $x$ in $X$ and $|\mathcal{P}_{x}|\leq\kappa$. Clearly, it need only to prove $cs_{\chi}(\mathcal{F}(X), A)\leq\kappa$ for each $A\in \mathcal{F}(X)$. Take any $F\in\mathcal{F}(X)$, and enumerate $F$ as $\{x_{1}, \cdots,x_{n}\}$. Set
$$\widehat{\mathcal{P}}_{F}=\{\langle  P_{x_{1}}, \cdots, P_{x_{n}}\rangle\cap\mathcal{F}(X):P_{x_{i}}\in \mathcal{P}_{x_{i}}, n\in \mathbb{N}\}.$$
By the proof of Proposition~\ref{pp}, we can see that $\widehat{\mathcal{P}}_{F}$ is a cs-network of $\mathcal{F}(X)$. Moreover, it is obvious that $|\widehat{\mathcal{P}}_{F}|\leq\kappa$. Therefore, $$cs_{\chi}(\mathcal{F}(X), A)\leq\kappa.$$
\end{proof}

However, the equalities (3), (4), (7) and (8) hold if $cs(X)=cs^{\ast}(X)=cs\omega(X)=cs^{\ast}\omega(X)=\aleph_{0}$. First, we need a lemma.

\begin{lemma}\label{l66}
If $X$ is a first-countable space and has a countable $cs$-network $\mathcal{P}$, then for each $x\in X$ and any open neighborhood $U$ of $x$ there exists a finite subfamily $\mathcal{F}\subset \mathcal{P}$ such that $x\in\mbox{Int}(\bigcup\mathcal{F})\subset U$.
\end{lemma}

\begin{proof}
Let $\mathcal{P}=\{P_{n}: n\in\mathbb{N}\}$. First, we prove that for each $x\in X$ there exists a finite subfamily $\mathcal{F}\subset \mathcal{P}$ such that $x\in\mbox{Int}(\bigcup\mathcal{F})$. Suppose not, there exists $x\in X$ such that $x\not\in\mbox{Int}(\bigcup\mathcal{F})$ for any finite subfamily $\mathcal{F}\subset \mathcal{P}$. Let $(P_{n})$ enumerate $\{P\in\mathcal{P}: x\in P\}$, and let $(U_{n})$ be a
countable decreasing base at $x$. Pick $x_{n}\in U_{n}\setminus\bigcup_{i=1}^{n}P_{i}$. Clearly, the set $\{x_{n}\}$ converges to $x$. Since $\mathcal{P}$ is a $cs$-network, there exist $N\in\mathbb{N}$ and $m\in \mathbb{N}$ such that $\{x_{n}: n>N\}\subset P_{m}$; however, $x_{n}\not\in P_{m}$ for any $n>m$, which is a contradiction. Take any open neighborhood $U$ of $x$. Then it is obvious that there exists a finite subfamily $\mathcal{F}\subset \mathcal{P}$ such that $x\in\mbox{Int}(\bigcup\mathcal{F})\subset U$.
\end{proof}

\begin{proposition}\label{ppp}
If $\mathcal{N}$ is a countable $cs$-network of a space $X$ and $\mathcal{N}$ is closed under finite unions, then $$\mathcal{B}=\{\langle B_{1}, ..., B_{n}\rangle:B_{i}\in \mathcal{N}, i\leq n, n\in\mathbb{N}\}$$ is a countable $cs$-network of $\mathcal{S}(X)$.
\end{proposition}

\begin{proof}
Obviously, $\mathcal{B}$ is countable. We need to prove that $\mathcal{B}$ is a $cs$-network of $\mathcal{S}(X)$. Suppose a sequence $\{A_{n}\}_{n\in \mathbb{N}}\subset \mathcal{S}(X)$ converges to a point $S=\bigcup _{n=1}^{k} S_{n}$ in $\mathcal{S}(X)$. Put $K=A\cup\bigcup\{A_{n}: n\in\mathbb{N}\}$. From Lemma~\ref{16} and (ii) of (2) in Lemma~\ref{18}, $K$ is countable compact metrizable. Moreover, $\mathcal{K}=\{P\cap K: P\in\mathcal{N}\}$ is a countable $cs$-network in $K$, and $\mathcal{S}(K)$ is a subspace of $\mathcal{S}(X)$. Let $\widehat{\mathcal{U}}$ be a neighborhood of $A$ in $\mathcal{S}(X)$. From Lemmas~\ref{11} and~\ref{l66}, there exist $B_{1}, \cdots, B_{k}, \cdots, B_{N}\in\mathcal{N}$ satisfy the following conditions:

\smallskip
(i) $A\in\langle \mbox{Int}_{K}(B_{1}\cap K), \cdots, \mbox{Int}_{K}(B_{k}\cap K), \cdots, \mbox{Int}_{K}(B_{N}\cap K)\rangle$;

\smallskip
(ii)  $\langle B_{1}, \cdots, B_{k}, \cdots, B_{N}\rangle\subset \widehat{\mathcal{U}}$;

\smallskip
(iii) $s_{i}\in \mbox{Int}_{K}(B_{i}\cap K)$ for each $i\leq k$.

\smallskip
Then $\bigcup_{i=1}^{N}\mbox{Int}_{K}(B_{i}\cap K)$ is an open neighborhood of $A$ in $K$, and the sequence $\{A_{n}\}_{n\in \mathbb{N}}\subset \mathcal{S}(K)$ converges to a point $S$ in $\mathcal{S}(K)$. Hence it follows from Lemma~\ref{18} that there exists $N_{1}\in\mathbb{N}$ such that $A_{n}\subset \bigcup_{i=1}^{N}\mbox{Int}_{K}(B_{i}\cap K)$ for any $n>N_{1}$. Moreover, from (i) of (2) of Lemma~\ref{18}, there exists $N_{2}\in\mathbb{N}$ such that $A_{n}\cap \mbox{Int}_{K}(B_{i}\cap K)\neq\emptyset$ for any $n>N_{2}$ and $i=1, \cdots, N$. Put $N_{3}=\max\{N_{1}, N_{2}\}$. Then for any $n>N_{3}$, we have  $A_{n}\in \langle \mbox{Int}_{K}(B_{1}\cap K), \cdots, \mbox{Int}_{K}(B_{k}\cap K), \cdots, \mbox{Int}_{K}(B_{N}\cap K)$, hence $$A_{n}\in\langle B_{1}, \cdots, B_{k}, \cdots, B_{N}\rangle\subset\widehat{\mathcal{U}}.$$
\end{proof}

\begin{theorem}
For any space $X$, the following are equivalent:
\begin{enumerate}
\item $cs^{\ast}\omega(X)\leq\aleph_{0}$;

\smallskip
\item $cs\omega(X)\leq\aleph_{0}$;

\item $cs^{\ast}\omega(\mathcal{S}(X))\leq\aleph_{0}$;

\smallskip
\item $cs\omega(\mathcal{S}(X))\leq\aleph_{0}$.
\end{enumerate}
\end{theorem}

\begin{proof}
Clearly, (4) $\Rightarrow$ (3) and (3) $\Rightarrow$ (1). It is well-known that a space $X$ has a countable $cs$-network if $X$ has a countable $cs^{\ast}$-network, see \cite[Lemma 2.2]{S2005}. Therefore, we have (1) $\Rightarrow$ (2). Then from Proposition~\ref{ppp} we conclude that (2) $\Rightarrow$ (4).
\end{proof}

Moreover, we have the following result by a similar proof of Proposition~\ref{ppp}.

\begin{theorem}
For any space $X$, the following are equivalent:
\begin{enumerate}
\item $X$ is $cs^{\ast}$-first-countable;

\smallskip
\item $X$ is $cs$-first-countable;

\item $\mathcal{S}(X)$ is $cs^{\ast}$-first-countable;

\smallskip
\item $\mathcal{S}(X)$ is $cs$-first-countable.
\end{enumerate}
\end{theorem}

Finally, we discuss the relations of $n\omega(X)$ and $cs\omega(X)$ of the hyperspace $\mathcal{S}(X)$.

\begin{theorem}\label{tt5}
For any space $X$, $n\omega(\mathcal{S}(X))=cs\omega(X)$.
\end{theorem}

\begin{proof}
Assume $n\omega(\mathcal{S}(X))=\kappa$, and assume that $\mathscr{P}$ is a network of $\mathcal{S}(X)$ with $|\mathscr{P}|=\kappa$. Put $$\mathscr{B}=\{\bigcup \widehat{P}: \widehat{P}\in\mathscr{P}\}.$$ Then $\mathscr{B}$ is a $cs$-network of $X$. Indeed, take any $x\in X$ and open neighborhood $U$ of $x$. If a sequence $\{x_{n}\}$ converges to $x$ in $X$, then there exists $N\in\mathbb{N}$ such that $K=\{x\}\cup\{x_{n}: n>N\}\subset U$. Clearly, $K\in \mathcal{S}(X)$ and $K\in\langle U\rangle$, hence there exists $\widehat{P}\in\mathscr{P}$ such that $K\in\widehat{P}\subset\langle U\rangle$, hence $K\subset\bigcup\widehat{P}\subset U$. Then $\mathscr{B}$ is a $cs$-network of $X$, and $|\mathscr{B}|\leq\kappa$, thus $cs\omega(X)\leq\kappa$.

 Assume $cs\omega(X)=\kappa$, and assume that $\mathscr{C}$ is a $cs$-network of $X$ with $|\mathscr{C}|=\kappa$. Without loss of generality, we may assume that $\mathscr{C}$ is closed under the finite unions. Put $$\mathscr{D}=\{\langle\mathscr{C}^{\prime}\rangle: \mathscr{C}^{\prime}\in\mathscr{C}^{<\omega}\}.$$ We claim that $\mathscr{D}$ is a network of $\mathcal{S}(X)$. Indeed, take any $S\in \mathcal{S}(X)$ and open neighborhood $\widehat{U}$ of $S$ in $\mathcal{S}(X)$. Then there exist disjoint open sets $V_{1}, \cdots, V_{N}$ of $X$ satisfying the conditions of Lemma~\ref{11} and $S\in\langle V_{1}, \cdots, V_{n}\rangle\subset\widehat{U}$. Since $\mathscr{C}$ is a $cs$-network of $X$ and it is closed under finite unions, for each $i\in\{1, \cdots, N\}$ we can find $P_{i}\in\mathscr{C}$ such that $S\cap V_{i}\subset P_{i}\subset V_{i}$. Then $$S\in\langle P_{1}, \cdots, P_{n}\rangle\subset \langle V_{1}, \cdots, V_{n}\rangle\subset\widehat{U}.$$ Clearly, $|\mathscr{D}|=\kappa$, thus $n\omega(X)\leq\kappa$.
\end{proof}

\begin{theorem}\label{tt4}
Let $X$ be a regular space, then the following statements are equivalent:
\begin{enumerate}
\item $\mathcal{S}(X)$ is cosmic;

\smallskip
\item $\mathcal{S}(X)$ is an $\aleph_{0}$-space;

\smallskip
\item $X$ is an $\aleph_{0}$-space.
\end{enumerate}
\end{theorem}

\begin{proof}
By \cite[Theorem 1.1]{IN}, we only need to prove (1) $\Rightarrow$ (3). Assume that $\mathcal{S}(X)$ is cosmic, then it follows from Theorem~\ref{tt5} that $X$
has a countable $cs$-network of $\mathcal{S}(X)$. By \cite{F1984}, $X$ is an $\aleph_{0}$-space.
\end{proof}

It is easy to see that $n\omega(X)=n\omega(\mathcal{F}(X))$ for any space $X$. However, the following example shows that $n\omega(X)=n\omega(\mathcal{S}(X))$ does not hold for a space $X$.

\begin{example}
There exists a zero-dimensional cosmic space $X$ such that $\mathcal{S}(X)$ is not a cosmic space.
\end{example}

\begin{proof}
Let $X$ be the space in \cite[Example]{Kenichi Tamano}. Then $X$ is a zero-dimensional cosmic space. However, $X$ is not a $\aleph_{0}$-space since it is not a $\mu$-space. Then, it follows from Theorem~\ref{tt4} that $\mathcal{S}(X)$ is not a cosmic space.
\end{proof}

\smallskip
\section{Rank $k$-diagonal of $\mathcal{S}(X)$}
In this section, we mainly discuss the rank $k$-diagonal of $\mathcal{S}(X)$. First of all, we recall the following concept of rank $k$-diagonal. Let $A$ be a subset of a space $X$, let $\gamma$ be a family of subsets of $X$, and let
$$\mbox{st}(A, \gamma)=\bigcup\{U\in\gamma: U\cap A\not =\emptyset\}.$$ We also put $\mbox{st}^{0}(A, \gamma)=A$ and $\mbox{st}^{n+1}(A, \gamma)=\mbox{st}(\mbox{st}^{n}(A, \gamma), \gamma)$ for any $n\in\mathbb{N}$.

\begin{definition}\cite{AB2006}
A {\it diagonal sequence of rank $k$} on a space $X$, where $k\in \omega$, is a countable family $\{\gamma_{n}: n \in\omega\}$ of open coverings of $X$ such that $$\{x\}=\bigcap\{\mbox{st}^{k}(x, \gamma_{n}):n\in \omega\}$$ for each $x\in X$. We say that $X$ has a {\it rank $k$-diagonal}, where $k\in\omega$, if there is a diagonal sequence $\{\gamma_{n}: n\in\omega\}$ on $X$ of rank $k$.
\end{definition}

Obviously, each space with a rank $k+1$-diagonal has a $k$-diagonal, and rank 1-diagonal is equivalent to $G_{\delta}$-diagonal. It is well known that there exists a space with a $G_{\delta}$-diagonal such that  $\mathcal{K}(X)$ does not have a $G_{\delta}$-diagonal. Therefore, it is natural to ask the following question.

\begin{question}\label{q2}
If $X$ is any space with a rank $k$-diagonal for some $k\in\mathbb{N}$, does $\mathcal{S}(X)$ have a rank $k$-diagonal?
\end{question}

We will give a negative answer to Question~\ref{q2}, and prove that $\mathcal{F}(X)$ has a rank $k$-diagonal if $X$ is space with a rank $k$-diagonal.

\begin{theorem}\label{17}
For $k\in\mathbb{N}$, a space $X$ has a rank $k$-diagonal if and only if $\mathcal{F}(X)$ has a rank $k$-diagonal.
\end{theorem}

\begin{proof}
Clearly, it suffices to prove the necessity. Let $\{\gamma_{n}: n\in \omega\}$ be a diagonal sequence with rank $k$ in $X$. Without loss of generality, we may assume that $\gamma_{n+1}$ refines $\gamma_{n}$ for each $n\in\omega$. For any ~$n\in\omega$, set $$\mathcal{A}_{n}=\{\langle U_{1},...,U_{m}\rangle\cap\mathcal{F}(X): U_{1},...,U_{m}\in \gamma_{n}, m\in\omega\}.$$ Obviously, $\mathcal{A}_{n}$ is an open cover of $\mathcal{F}(X)$ and $\mathcal{A}_{n+1}$ refines $\mathcal{A}_{n}$ for each $n\in\omega$. Next we prove that $\{\mathcal{A}_{n}: n\in \omega\}$ is a diagonal sequence of $\mathcal{F}(X)$ with rank $k$.

In fact, we will prove that $\{F\}=\bigcap\{\mbox{st}^{k}(F, \mathcal{A}_{n}):n\in\omega\}$ for any $F\in \mathcal{F}(X)$. Fix an arbitrary $F\in\mathcal{F}(X)$. It follows from the definition that $F\in\bigcap\{\mbox{st}^{k}(F,\mathcal{A}_{n}):n\in\omega\}$. Suppose there is $K\in \mathcal{F}(X)\setminus\{F\}$ such that $K\in \bigcap\{\mbox{st}^{k}(F, \mathcal{A}_{n}):n\in\omega\}$.

\smallskip
In order to obtain a contradiction, we divide the proof into the following two cases:

\smallskip
{\bf Case 1}: $K\backslash F\not =\emptyset$.

\smallskip
Then there is $y\in K\backslash F$, thus $y\not\in F $. Because $\{\gamma_{n}: n\in \omega\}$ is a diagonal sequence with rank $k$ in $X$, it follows that $y\not\in \bigcap\{\mbox{st}^{k}(x, \gamma_{n}): n\in\omega\}$ for each $x\in F$. Then, for each $x\in F$ there exists $n_{x}\in\omega$ such that $y\not\in \mbox{st}^{k}(x, \gamma_{n_{x}})$. Put $n_{y}=\max\{n_{x}: x\in F\}$. Then, since $\gamma_{n+1}$ refines $\gamma_{n}$ for each $n\in\omega$, we have that $\mbox{st}^{k}(x, \gamma_{n_{y}})\subset\mbox{st}^{k}(x, \gamma_{n_{x}})$ for any $x\in F$. Therefore, $y\not\in \mbox{st}^{k}(x, \gamma_{n_{y}})$ for any $x\in F$. Then any element of $\gamma_{n_{y}}$, which contains the point $y$, does not intersect with $\mbox{st}^{k-1}(x, \gamma_{n_{y}})$.

Since $K\in \bigcap\{\mbox{st}^{k}(F, \mathcal{A}_{n}):n\in\omega\}$, for any $n\in \omega$ we see that $K\in \mbox{st}^{k}(F, \mathcal{A}_{n})$, then there exists a subset $\{\widehat{\mathcal{U}_{n}(i)}: i=1, \cdots, k\}$ of $\mathcal{A}_{n}$ satisfying the following conditions:
\begin{enumerate}
\item $\widehat{\mathcal{U}_{n}(i)}\cap \mbox{st}^{i-1}(F,\mathcal{A}_{n})\not=\emptyset$ for each $i=1,\cdots, k$, where
    $\mbox{st}^{0}(F,\mathcal{A}_{n})=\{F\}$;

\smallskip

\item $\widehat{\mathcal{U}_{n}(i)}\cap \widehat{\mathcal{U}_{n}(i-1)}\not=\emptyset$ for each $i=2,\cdots, k$;

\smallskip
\item $K\in \widehat{\mathcal{U}_{n}(k)}$;

\smallskip
\item $\widehat{\mathcal{U}_{n}(i)}=\langle U_{n}(i, 1), \cdots, U_{n}(i, m_{i, n})\rangle\cap\mathcal{F}(X)$ for each $i=1, \cdots, k$, where $U_{n}(i, j)\in \gamma_{n}$ for each $1\leq j\leq m_{i, n}$.
\end{enumerate}

Because $K\in \widehat{\mathcal{U}_{n_{y}}(k)}$, without loss of generality we may assume that $y\in U_{n_{y}}(k, 1)$. From the above constructions (1) and (2), we also may assume that $$U_{n_{y}}(i, 1)\cap U_{n_{y}}(i-1, 1)\neq\emptyset$$ for each $i=1, \cdots, k$. Therefore, $U_{n_{y}}(k, 1)\cap\mbox{st}^{k-1}(x, \gamma_{n_{y}})\neq\emptyset$ for any $x\in U_{n_{y}}(1, 1)\cap F\neq\emptyset$, which is a contradiction.

\smallskip
{\bf Case 2}: $K\backslash F=\emptyset$.

Thus there is $z\in F\backslash K$, that is, $x\not =z$ for each $x\in K$. Because $\{\gamma_{n}:n\in \omega\}$ is the diagonal sequence with rank $k$ in $X$, for each $x\in K$ we can conclude that $z\not\in \bigcap\{\mbox{st}^{k}(x,\gamma_{n}):n\in\omega\}$, hence there exists $n_{x}\in\omega$ such that $z\not\in \mbox{st}^{k}(x,\gamma_{n_{x}})$, then any element of $\gamma_{n_{x}}$ containing $z$ does not intersect $\mbox{st}^{k-1}(x, \gamma_{n_{x}})$. From $K\in st^{k}(F, \mathcal{A}_{n})$ for any $n\in \omega$, it follows that there is $\widehat{\mathcal{U}_{n}(k)}\in\mathcal{A}_{n}$ such that $K\in \widehat{\mathcal{U}_{n}(k)}$ and $\widehat{\mathcal{U}_{n}(k)}\bigcap st^{k-1}(F, \mathcal{A}_{n})\not =\emptyset$. Since $\widehat{\mathcal{U}_{n}(k)}\bigcap st^{k-1}(F, \mathcal{A}_{n})\not =\emptyset$, then there is $\widehat{\mathcal{U}_{n}(k-1)}\in\mathcal{A}_{n}$ such that $ \widehat{\mathcal{U}_{n}(k)}\bigcap \widehat{\mathcal{U}_{n}(k-1)}\not =\emptyset$ and $ \widehat{\mathcal{U}_{n}(k)}\bigcap st^{k-1}(F, \mathcal{A}_{n})\not =\emptyset$. Since $k$ is finite, there is $\widehat{\mathcal{U}_{n}(2)}\in\mathcal{A}_{n}$ such that $\widehat{\mathcal{U}_{n}(2)}\bigcap st(F, \mathcal{A}_{n})\not =\emptyset$. Therefore, there exists $\widehat{\mathcal{U}_{n}(1)}\in\mathcal{A}_{n}$ such that $F\in \widehat{\mathcal{U}_{n}(1)}$ and $\mathcal{U}_{n}(2)\bigcap \mathcal{U}_{n}(1)\not =\emptyset$. Hence $F\in \mbox{st}^{k}(K, \mathcal{A}_{n})$ for any $n\in \omega$, then we will obtain a contradiction by a similar proof Case (1).
\end{proof}

The following corollary is the particular case of Theorem~\ref{17} when $k=1$.

\begin{corollary}
A space $X$ has a $G_{\delta}$-diagonal if and only if $\mathcal{F}(X)$ has a $G_{\delta}$-diagonal.
\end{corollary}

\smallskip
However, the conclusion of Theorem~\ref{17} does not valid in the hyperspaces of finite unions of convergent sequences by the following example. First, we recall a concept.

Assume we are given two disjoint spaces $X$ and $Y$ and a continuous mapping $f: M\rightarrow Y$ defined on a closed subset $M$ of the space $X$. Let $E$ be an equivalence relation on the sum $X\oplus Y$ corresponding to the decomposition of $X\oplus Y$ into the one-point sets $\{x\}$, where $x\in X\setminus M$, and sets of the form $\{y\}\cup f^{-1}(y)$, where $y\in Y$. The quotient space $(X\oplus Y)/E$ is call the {\it adjunction space} \cite{E1989} determined by $X, Y$ and $f$ and is denoted by $X\bigcup_{f}Y$.

\begin{example}\label{15}
There exists a space $X$ which has a rank $2$-diagonal, but the space $\mathcal{S}(X)$ does not have $G_{\delta}$-diagonal.
\end{example}

\begin{proof}
Let $X=X_{1}\bigcup X_{2}$, where $X_{1}=\mathbb{R}\times (\{0\}\cup\{1/n: n\in \mathbb{N}\}), X_{2}=\mathbb{R}\times \{-1\}.$  Now $X$ is endowed with a topology as follows: for any $n\in \mathbb{N}$, $\mathbb{R}\times \{1/n\}$ is exactly a copy of $\mathbb{R}^{*}$, where $\mathbb{R}^{*}$ is the real line $\mathbb{R}$ equipped with the following topology:

\smallskip
(a) Each point of the irrational set $\mathbb{P}$ is an isolated point in $\mathbb{R}^{*}$;

\smallskip
(b) The neighborhood basis of each point $r\in\mathbb{Q}$ is the following family $$\{\{r\}\cup (r-\varepsilon,r+\varepsilon)\cap \mathbb{P}:\varepsilon >0\}.$$

For each $p\in \mathbb{R}\times \{0, -1\}$, let $\{N(p, \varepsilon): \varepsilon>0\}$ be the neighborhood base of $p$ in $X$, where

\smallskip
 \begin{enumerate}
\item if $p=(x, 0) \mbox{ and }x\in \mathbb{Q}$, then $$N(p, \varepsilon)=\{p\}\cup \{(x',y')\in X: 0< y'<|x'-x|,|x'-x|<\varepsilon\};$$
\item if $p=(x, 0) \mbox{ and }x\in \mathbb{P}$, then $$N(p, \varepsilon)=\{p\}\cup \{(x,y')\in X: 0< y'<\varepsilon\};$$
\item if $p=(x,-1) \mbox{ and }x\in \mathbb{Q}$, then $$N(p, \varepsilon)=\{p\}\cup \{(x',y')\in X: |x'-x|< y'<\varepsilon\};$$
\item if $p=(x,-1) \mbox{ and }x\in \mathbb{P}$, then $N(p, \varepsilon)=\{p\}$.
\end{enumerate}

 \smallskip
Define the function f: $X_{2}\rightarrow \mathbb{R}^{*},$ where $f((x, -1))=x, (x, -1)\in X_{2}$. Clearly, $f$ is continuous. Indeed, take any $(x, -1)\in X_{2}$. If $x\in\mathbb{P}$, then $f$ is continuous at $(x, -1)$ since $\{x\}$ and $\{(x, -1)\}$ are open in $\mathbb{R}^{*}$ and $X_{2}$ respectively; if $x\in\mathbb{Q}$, then $N((x, -1), \varepsilon)\cap X_{2}=\{(x, -1)\}$ for any $\varepsilon<1$, hence $f$ is continuous at $(x, -1)$. Now let $Z=X\bigcup_{f}\mathbb{R}^{*}$ be the adjuction space. Obviously, $Z$ is a $T_{2}$-space. We claim that $Z$ has a rank $2$-diagonal.

\smallskip
 {\bf Claim:} The space $Z$ has a rank 2-diagonal.

\smallskip
Indeed, for any $n\in\mathbb{N}$, let $$\mathcal{A}_{n}=\left\{\left((x-\frac{1}{n}, x+\frac{1}{n})\cap (\mathbb{P}\setminus\{x\})\right)\times\{\frac{1}{m}\}: x\in \mathbb{Q}, m\in\mathbb{N}\right\}\cup\left\{\{(x, \frac{1}{m})\}: x\in\mathbb{P}, m\in\mathbb{N}\right\},$$
 $$\mathcal{B}_{n}=\left\{N(p, \frac{1}{n}): p\in \mathbb{R}\times \{0\}\right\},$$
 $$\mathcal{V}_{n}=\left\{\left((x-\frac{1}{n}, x+\frac{1}{n})\cap (\mathbb{P}\setminus\{x\})\right)\cup \left(N(p, \frac{1}{2n})\setminus\{p\}\right): p=(x, -1), x\in\mathbb{Q}\right\}\cup\left\{\{x\}: x\in\mathbb{P}\right\}$$and
 $$\mathcal{W}_{n}=\mathcal{A}_{n}\cup\mathcal{B}_{n}\cup\mathcal{V}_{n}.$$ Then $\mathcal{W}_{n}$ is a rank 2-diagonal sequence of $Z$. In fact, take any $p, q\in Z, p\not=q$. In order to prove that $p\not\in\mbox{st}^{2}(q, \mathcal{W}_{N})$ for some $N\in\mathbb{N}$, we divide the proof into the following three cases.

\smallskip
{\bf case1:} $p, q\in \mathbb{R}^{*}$.

\smallskip
Since $p\neq q$, there is $n\in\mathbb{N}$ such that $|p-q|>\frac{1}{n}$, thus it is easy to verify that $p\not\in\mbox{st}^{2}(q, \mathcal{W}_{4n})$.

\smallskip
{\bf case 2:} $p, q\in X_{1}$.

\smallskip
Assume that there exists at least one point of $\{p, q\}$ which does not belong to $\mathbb{R}\times \{0\}$. Without loss of generality, we may assume that $p\not\in\mathbb{R}\times \{0\}$, then there exits $m\in\mathbb{N}$ such that $p\in\mathbb{R}\times \{\frac{1}{m}\}$. If $q\not \in\mathbb{R}\times \{\frac{1}{m}\}$, then from the above construction it follows that $q\not\in\mbox{st}^{2}(p, \mathcal{W}_{m})\subset\mathbb{R}\times \{\frac{1}{m}\},$ thus $p\not\in\mbox{st}^{2}(q, \mathcal{W}_{m});$ If $q\in\mathbb{R}\times \{\frac{1}{m}\}$, then let $p=(x_{1}, \frac{1}{m}), q=(x_{2}, \frac{1}{m})$, hence there exists $m_{1}\in\mathbb{N}$ such that $|x_{1}-x_{2}|>\frac{1}{m_{1}}$. Therefore, $p\not\in\mbox{st}^{2}(p, \mathcal{W}_{4m_{1}})$.

Put $p, q\in\mathbb{R}\times \{0\}$, and let $p=(y_{1}, 0), q=(y_{2}, 0)$. Then there is $m_{2}\in\mathbb{N}$ such that $|y_{1}-y_{2}|>\frac{1}{m_{2}}$, hence $p\not\in\mbox{st}^{2}(p, \mathcal{W}_{4m_{2}})$.

\smallskip
{\bf case 3:} $p\in X_{1}, q\in \mathbb{R}^{*}$.

\smallskip
If there exists $n\in\mathbb{N}$ such that $p\in \mathbb{R}\times\{\frac{1}{n}\}$, then it is obvious that $$q\not\in st^{2}(p, \mathcal{W}_{n})\subset \mathbb{R}\times\{\frac{1}{n}\}.$$
Therefore, assume that $p=(x, 0)$. If $x=q\in \mathbb{Q}$, then for any $n\in\mathbb{N}$ we have $$\mbox{st}(q, \mathcal{W}_{n})=((q-\frac{1}{n}, q+\frac{1}{n})\cap (\mathbb{P}\setminus\{q\}))\cup (N(p, \frac{1}{kn}).$$ Put $$O_{n}=((q-\frac{1}{n}, q+\frac{1}{n})\cap (\mathbb{P}\setminus\{q\}))\cup N(p, \frac{1}{2n}).$$ Hence $$\mbox{st}(\mbox{st}(q, \mathcal{W}_{n}), \mathcal{W}_{n})=\mbox{st}(O_{n}, \mathcal{W}_{n}).$$ By the definition of the topology of $X$, it is easy to verify $p\not\in\mbox{st}(O_{n}, \mathcal{W}_{n})$, that is, $p\not\in\mbox{st}^{2}(q, \mathcal{W}_{n}).$  If $x=q\in \mathbb{P}$, we have that $\mbox{st}^{2}(q, \mathcal{W}_{n})=\{q\}$, hence $p\not\in\mbox{st}^{2}(q, \mathcal{W}_{n}).$

Finally we prove that $S(Z)$ does not have any $G_{\delta}$-diagonal. Suppose not, then $S(Z)$ has a $G_{\delta}$-diagonal sequence $\{\widehat{\mathcal{U}}(n):n\in N\}$. Put $$K(s)=\{(s, 0)\}\cup\{(s, 1/n): n\in \mathbb{N}\}, L(s)=K(s)\cup \{s\}$$ for each $s\in \mathbb{P}\subset \mathbb{R}$. Since $K(s), L(s)\in S(Z), $ and $K(s)\not= L(s)$, $n(s)\in \mathbb{N}$ for each $s\in \mathbb{P}$ such that $K(s)\not \in \mbox{st}(L(s), \widehat{\mathcal{U}}(n(s)))$. Because $\mathbb{R}$ is the second category, there is $n\in \mathbb{N}$ such that $\mbox{Int}(\mbox{Cl}\{s\in \mathbb{P}:n(s)=n\})\not=\emptyset$, where Int and C1 are in the sense of the usual topology. Then there is $r\in \mathbb{Q}$ such that $r\in \mbox{Int}(\mbox{Cl}\{s\in \mathbb{P}: n(s)=n\})$. Set $$R(r)=\{r\}\cup\{(r,0)\}\cup\{(r,1/n):n\in N\}.$$ Hence it follows that $R(r)\in \mathcal{S}(X)$. Since $\{\widehat{\mathcal{U}}(n): n\in \mathbb{N}\}$ is a $G_{\delta}$-diagonal sequence of ~$S(Z)$, we can conclude that there exists $\widehat{U}\in \widehat{\mathcal{U}}(n)$ such that $R(r)\in \widehat{U}$. From the definition of the topology of $X$, there exists $s\in \mathbb{P}$ such that $n(s)=n, K(s), L(s)\in \widehat{U}$, which is a contradiction with $K(s)\not \in \mbox{st}(L(s), \widehat{\mathcal{U}}(n(s)))$.
\end{proof}

If $X$ is a submetrizable space, then $\mathcal{K}(X)$ is also a submetrizable space, and since each submetrizable space has a rank $k$-diagonal for each $k\in\mathbb{N}$, it follows that $\mathcal{K}(X)$ has a rank $k$-diagonal for each $k\in\mathbb{N}$. Therefore, by example~\ref{15}, it is natural to pose the following question.

\begin{question}
For any $k>2$, if $X$ has a rank $k$-diagonal, then does $\mathcal{S}(X)$ have a rank $k$-diagonal?
\end{question}

 \smallskip
\section{Some generalized metric properties on $\mathcal{S}(X)$}
In this section, we discuss the relation of copies of $S_{2}$ and $S_{\omega}$ and the $\gamma$-space property on $\mathcal{S}(X)$.
First, we give a relation of copies of $S_{2}$ and $S_{\omega}$ of the hyperspace $\mathcal{S}(X)$.

\begin{proposition}\label{ming ti 5}
If $X$ contains a closed copy of $S_2$, then $\mathcal{S}(X)$
contains a closed copy of $S_{\omega}$.
\end{proposition}

\begin{proof}
Let $$F=\{y\}\cup\{y_n: n\in \mathbb{N}\}\cup\{y_i(n): i, n\in\mathbb{N}\}$$ be a
closed copy of $S_2$ in $X$, where $y_i(n)\to y_n$ as $i\to\infty$
for each $n\in \mathbb{N}$, $y_n\to y$ as $n\to\infty$ and for any
$f\in \mathbb{N}^\mathbb{N}$, $\{y_i(n): i\leq f(n), n\in
\mathbb{N}\}$ is discrete.

Let $K=\{y\}\cup \{y_n: n\in \mathbb{N}\}$, and for each $i,
n\in\mathbb{N}$, let $K_i(n)=\{y_i(n)\}\cup K$. Obviously, all $K,
K_i(n)\in \mathcal{S}(X)$, and $K_i(n)\to K$ as $i\to\infty$ for
each $n\in\mathbb{N}$. Let $$\mathcal{A}=\{K\}\cup\{K_i(n): i,
n\in\mathbb{N}\}.$$ We claim that $\mathcal{A}$ is a closed copy of $S_\omega$
in $\mathcal{S}(X)$.

Indeed, since $F$ is closed in $X$, it follows that $\mathcal{S}(F)$
is closed in $\mathcal{S}(X)$. In order to prove that
$\mathcal{A}$ is closed, it suffices to show that $\mathcal{A}$ is closed in $\mathcal{S}(F)$.

For $H\in \mathcal{S}(F)\setminus \mathcal{A}$, we need to find a neighborhood
$\widehat{U}$ of $H$ in $\mathcal{S}(F)$ such that $\widehat{U}\cap
\mathcal{A}=\emptyset$. We consider the following two cases.

\smallskip
{\bf Case 1} $K\setminus H\neq\emptyset$.

\smallskip
Pick $z\in K\setminus H$, and for each $x\in H$, take an open
neighborhood $V_x$ of $x$ in $X$ such that $z\notin V_x$. Then there
are finitely many $V_{x_i}(i\leq k)$ such that $H\subset
 \bigcup_{i\leq k}V_{x_i}$. Let $\widehat{U}=\langle V_{x_1}, ...,
 V_{x_k}\rangle$. Then it is easy to check that $\widehat{U}\cap
 \mathcal{A}=\emptyset$.

\smallskip
{\bf Case 2} $K\subset H$.

\smallskip
Since $H\notin \mathcal{A}$, it follows that $|H\cap\{y_i(n): i, n\in
\mathbb{N}\}|\geq 2$. Pick $z_1, z_2\in H\cap\{y_i(n): i, n\in
\mathbb{N}\}$, and let $\widehat{U}=\langle F, \{z_1\},
\{z_2\}\rangle$. Then $\widehat{U}$ is an open neighborhood of $H$ in $\mathcal{S}(F)$ and
$\widehat{U}\cap \mathcal{A}=\emptyset$.

Therefore, $\mathcal{A}$ is closed in $\mathcal{S}(X)$.

\smallskip
Next we prove that $\mathcal{A}$ is a copy of $S_\omega$. Indeed, since $K_{i}(n)\in\langle F, \{y_{i}^{n}\}\rangle$ and $\mathcal{A}\cap \langle F, \{y_{i}^{n}\}\rangle=\{K_{i}(n)\}$ for each $i, n\in\mathbb{N}$, it follows that each point $K_{i}(n)$ is an isolated point in $\mathcal{A}$.
From \cite[page 48 of Example 1.8.7]{linbook}, it suffices to prove that, for any $f\in
\mathbb{N}^\mathbb{N}$, the set $\mathcal{B}=\{K_i(n): n\in \mathbb{N}, i\leq
f(n)\}$ is discrete in $\mathcal{S}(F)$. Since $\mathcal{A}$ is closed in
$\mathcal{S}(X)$, we only show that $\mathcal{B}$ is discrete in $\mathcal{A}$. Take an
arbitrary $H\in \mathcal{A}$. Then it suffices to find a neighborhood
$\widehat{U}$ of $H$ in $\mathcal{S}(F)$ such that $\widehat{U}$ intersects at most one element of
$\mathcal{B}$.
If $H\not\in \mathcal{B}$, then let $U=F\setminus \{y_i(n): n\in \mathbb{N},
i\leq f(n)\}$. Obviously, $U$ is open in $F$, thus put
$\widehat{U}=\langle U\rangle$. It is easy to see that $\widehat{U}\cap
\mathcal{B}=\emptyset$.  Now assume $H\in \mathcal{B}$, then $H=K\cup\{y_i(n)\}$ for
some $i, n\in\mathbb{N}$. Let $\widehat{U}=\langle F,
\{y_i(n)\}\rangle$. Hence it is easy to see that $|\widehat{U}\cap
\mathcal{B}|=1$.
\smallskip
In a word, $\mathcal{A}$ is a closed copy of $S_\omega$ in $\mathcal{S}(X)$.
\end{proof}

The following theorem generalizes a result in \cite{LF}.

A space $(X, \tau)$ is a {\it $\gamma$-space} if  there exists a function
$g: \omega\times X\to \tau_X$ such that (i) $\{g(n, x):
n\in\omega\}$ is a base at $x$ and $g(n+1, x)\subset g(n, x)$ for any $n\in\omega$; (ii) for each $n\in\omega$ and $x\in
X$, there exists $m\in\omega$ such that $y\in g(m, x)$ implies $g(m,
y)\subset g(n, x)$.

\begin{theorem}
A space $X$ is a $\gamma$-space if and only if $\mathcal{S}(X)$ is a
$\gamma$-space.
\end{theorem}

\begin{proof}
Assume that $X$ is a $\gamma$-space. Hence there exists a
$g$-function $g: \omega\times X\to \tau_X$ satisfying the definition
above. We also may assume that $g(n+1, x)\subset g(n, x)$ for each
$x\in X$ and $n\in \omega$. Define $G: \omega\times
\mathcal{S}(X)\to \tau_{\mathcal{S}(X)}$ by $$G(0, S)=\langle g(n_{0}^{S}, s_1),..., g(n_{0}^{S}, s_{k}), g(n_{0}^{S}, y(S, 0)), \cdots, g(n_{0}^{S}, y(S, r_{0}^{S}))\rangle$$ such that the family of open subsets $\{g(n_{0}^{S}, s_i): 1\leq i\leq k\}$ are disjoint, where $r_{0}^{S}\in\mathbb{N}$ and $$S\setminus\bigcup_{i=1}^{k}g(0, s_{i})=\{y(S, 0), \cdots, y(S, r_{0}^{S})\}.$$ For any $m\geq 1$, assume that
\begin{eqnarray}
G(m, S)&=&\langle g(n_{m}^{S}, s_1),..., g(n_{m}^{S}, s_{k}), g(n_{m}^{S}, y(S, 0)), \cdots, g(n_{m}^{S}, y(S, r_{0}^{S})),\nonumber\\
&& \ \ \ \ \ \ \ \ \ \ \ \ \ \ \ \ \ \ \ \ \ \ \ \ \ \ \ \ \ \ \ \cdots, g(n_{m}^{S}, y(S, r_{m}^{S}))\rangle\nonumber
\end{eqnarray}
satisfies the following conditions (1)-(4), where $r_{0}^{S}\leq r_{1}^{S}, \ldots, \leq r_{m}^{S}$ and $$S\setminus\bigcup_{i=1}^{k}g(n_{m}^{S}, s_{i})=\{y(S, 0), \cdots, y(S, r_{0}^{S}), \ldots, y(S, r_{m}^{S})\}.$$

\smallskip
(1) $n_{m}^{S}>n_{m-1}^{S}$.

\smallskip
(2) $\bigcup_{x\in S_{i}\cap g(n_{m-1}^{S}, s_i)}g(n_{m}^{S}, x)\subset g(n_{m-1}^{S}, s_{i})$ for each $1\leq i\leq k$ (This is possible from \cite[Theorem 10.6]{G1984}).

\smallskip
(3) For each $x\in A_{m-1}$, if $y\in g(n_{m}^{S}, x)$ then $g(n_{m}^{S}, y)\subset g(n_{m-1}^{S}, x)$ (This is possible since $X$ is a $\gamma$-space), where $$A_{m-1}=\{s_{i}: 1\leq i\leq k\}\cup\{y(S, 0), \cdots, y(S, r_{0}^{S}), \ldots, y(S, r_{m-1}^{S})\}.$$

\smallskip
(4) The family $\{\{g(n_{m}^{S}, s_i): 1\leq i\leq k\}\}\cup \{g(n_{m}^{S}, y(S, i)): i\leq r_{m-1}^{S}\}$ are disjoint.

\smallskip
For each $i\leq k$, put $S^{\prime}_{i}=S_{i}\cap g(n_{m}^{S}, s_{i})$. We construct $G(m+1, S)$ as follows.

Clearly, there exists a positive integer $N>n_{m}^{S}$ such that the family $$\{\{g(N, s_i): 1\leq i\leq k\}\}\cup \{g(N, y(S, i)): i\leq r_{m}^{S}\}$$ are disjoint. Then, for each $i\leq k$, it follows from \cite[Theorem 10.6]{G1984} that there exists a positive integer $n(i)>N$ such that $\bigcup_{x\in S^{\prime}_{i}}g(n(i), x)\subset g(n_{m}^{S}, s_{i})$. Moreover, since $X$ is a $\gamma$-space, for each $x\in A_{m}=\{s_{i}: 1\leq i\leq k\}\cup\{y(S, 0), \cdots, y(S, r_{0}^{S}), \ldots, y(S, r_{m}^{S})\}$, there exists a positive integer $m(x)>n_{m}^{S}$ such that if $y\in g(m(x), x)$ then $g(m(x), y)\subset g(n_{m}^{S}, x).$
Let $n_{m+1}^{S}=\max\left(\{n(i): 1\leq i\leq k\}\cup\{m(x): x\in A_{m}\}\right)$. Put
\begin{eqnarray}
G(m+1, S)&=&\langle g(n_{m+1}^{S}, s_1),..., g(n_{m+1}^{S}, s_{k}), g(n_{m+1}^{S}, y(S, 0)), \cdots, g(n_{m+1}^{S}, y(S, r_{m}^{S})),\nonumber\\
&& \ \ \ \ \ \ \ \ \ \ \ \ \ \ \ \ \ \ \ \ \ \ \ \ \ \ \ \ \ \ \ \ldots, g(n_{m+1}^{S}, y(S, r_{m+1}^{S}))\rangle,\nonumber
\end{eqnarray}
where $$S\setminus\bigcup_{i=1}^{k}g(n_{m+1}^{S}, s_{i})=\{y(S, 0), \cdots, y(S, r_{0}^{S}), \ldots, y(S, r_{m}^{S}), \ldots, y(S, r_{m+1}^{S})\}.$$Obviously, it easily check that $G(m+1, S)$ satisfies (1)-(4).

We claim that the function $G$
satisfies the definition above.

\smallskip
(i) For each $S=\bigcup _{n=1}^{k} S_{n}\in \mathcal{S}(X)$, the family
$\{G(m, S): m\in\mathbb{N}\}$ is a local base at
$S$ in $\mathcal{S}(X)$ and $G(m+1, S)\subset G(m, S)$ for any $m\in\omega$.

\smallskip
From the definition of $G$, we have $G(m+1, S)\subset G(m, S)$ for any $m\in\omega$. Now it only needs to prove that the family
$\{G(m, S): m\in\mathbb{N}\}$ is a local base at
$S$ in $\mathcal{S}(X)$. Indeed, take any open neighborhood $\widehat{U}$ of $S$ in $\mathcal{S}(X)$. Then there are disjoint open sets $V_{1}, \cdots, V_{k}, \cdots, V_{N}$ in $X$ which satisfy the conditions in Lemma~\ref{11} such that
$S\in\langle V_1,..., V_k, \cdots, V_{N}\rangle\subset \widehat{U}$. Obviously,
for each $1\leq i\leq  k$, there exists $n(i)$ such
that $g(j, s_{i})\subset V_i$ whenever $j\geq n(i)$; for each $x\in S\setminus\bigcup_{i=1}^{k}V_{i}$, there exist $n(x)$ and $k<i\leq N$ such
that $g(j, x)\subset V_i$ whenever $j\geq n(x)$.
Let $$r=\max\left(\{n(i): i=1, \cdots, k\}\cup \{n(x):
x\in S\setminus\bigcup_{i=1}^{k}V_{i}\}\right).$$Moreover, from \cite[Theorem 10.6]{G1984}, there exists $p\in\mathbb{N}$ such that $n_{p}^{S}>r$ and $$\bigcup_{x\in S\cap V_{i}}g(n_{p}^{S}, x)\subset V_{i}$$ for each $1\leq i\leq k$.
Then it follows that
\begin{eqnarray}
S\in G(p, S)&=&g(n_{p}^{S}, s_1),..., g(n_{p}^{S}, s_k), g(n_{p}^{S}, y(S, 0)), \cdots, g(n_{p}^{S}, y(S, r_{p}^{S}))\rangle\nonumber\\
&\subset&\langle V_1,..., V_k, \cdots, V_{N}\rangle\nonumber\\
&\subset&\widehat{U}.\nonumber
\end{eqnarray}

\smallskip
(ii) For any $S=\bigcup _{n=1}^{k} S_{n}\in \mathcal{S}(X)$ and $m\in\mathbb{N}$,
there exists $p\in\mathbb{N}$ such that $G(p, T)\subset G(m, S)$ for any $T\in G(p, S)$.

\smallskip
Take any $S=\bigcup _{n=1}^{k} S_{n}\in \mathcal{S}(X)$ and $m\in\mathbb{N}$. Since
$X$ is a $\gamma$-space, for each $$x\in A_{m+2}=\{s_{i}: i=1, \cdots, k\}\cup\{y(S, 0), \cdots, y(S, r_{0}^{S}), \ldots, y(S, r_{m}^{S}), \ldots, y(S, r_{m+2}^{S})\}$$ there exists a positive integer $m(x)>n_{m+2}^{S}$ such that $g(m(x), y)\subset g(n_{m+2}^{S},
x)$ if $y\in g(m(x), x)$. Put $p=\max\left\{m(x): x\in A_{m+2}\}\right\}.$
Take an arbitrary $T\in G(p, S)$. We prove that
$$G(p, T)\subset G(m, S).$$

Indeed, let $T=\bigcup _{n=1}^{l} T_{n}\in \mathcal{S}(X)$, where $T_{n}=\{t_{nj}\}_{j\in \mathbb{N}}\cup\{t_{n}\}$ and $t_{nj}\rightarrow t_{n}~(j\rightarrow\infty)$ for each $n\in\{1, ..., l\}$ such that $T_{t}\cap T_{n}=\emptyset$ for distinct $n, t\in\{1,...,l\}$. Then, from \cite[Lemma 2.3.1]{M}, we need to prove the following two claims.

\smallskip
{\bf Claim 1} $(\bigcup_{i=1}^{l}g(n_{p}^{T}, t_{i}))\cup (\bigcup _{j=1}^{r_{p}^{T}}g(n_{p}^{T}, y(T, j)))\subset \bigcup_{x\in A_{m}}g(n_{m}^{S}, x).$

\smallskip
For each $z\in D=\{t_{i}: 1\leq i\leq l\}\cup \{y(T, i): 1\leq i\leq r_{p}^{T}\}$, since $K\in G(p, S)$, there exists $x\in A_{m+2}$ or $r_{m+2}^{S}<j\leq r_{p}^{S}$ such that $z\in g(n_{p}^{S}, x)$ or $z\in g(n_{p}^{S}, y(S, j))$. We divide the proof into the following two cases.

\smallskip
{\bf Case 1} $z\in g(n_{p}^{S}, x)$ for some $x\in A_{m+2}$.

Clearly, $z\in g(n_{p}^{S}, x)\subset g(p, x)\subset g(m(s_{j}), x)$. Hence $g(m(s_{j}), z)\subset g(n_{m+2}^{S}, x)$, then $$g(n_{p}^{T}, z)\subset g(p, z)\subset g(m(s_{j}), z)\subset g(n_{m+2}^{S}, x)\subset g(n_{m+1}^{S}, y)\subset g(n_{m}^{S}, y)$$ for some $y\in A_{m}$.

\smallskip
{\bf Case 2} $z\in g(n_{p}^{S}, y(S, j))$ for some $r_{m+2}^{S}<j\leq r_{p}^{S}$.

Since $r_{m+2}^{S}<j\leq r_{p}^{S}$, there exists $m+2\leq v< p$ such that $r_{v}^{S}<j\leq r_{v+1}^{S}$. Then $z\in g(n_{p}^{S}, y(S, j))\subset g(n_{v+1}^{S}, y(S, j))$, hence $$z\in g(n_{v+1}^{S}, y(S, j))\subset g(n_{v}^{S}, s_{h})\subset g(n_{m+2}^{S}, s_{h})$$ for some $h\leq k$. Therefore, $g(n_{m+2}^{S}, z)\subset g(n_{m+1}^{S}, s_{h})$, then$$g(n_{p}^{T}, z)\subset g(p, z)\subset g(n_{m+2}^{S}, z)\subset g(n_{m}^{S}, s_{h}).$$

\smallskip
{\bf Claim 2} For each $x\in A_{m}$, there exists $y\in D$ such that $g(n_{p}^{T}, y)\subset g(n_{m}^{S}, x)$.

\smallskip
Indeed, since $T\in G(p, S)$, then there exists $j\leq l$ such that $T_{j}\cap g(n_{p}^{S}, x)\neq\emptyset$, hence $T_{j}\cap g(n_{m+1}^{S}, x)\neq\emptyset$. By the proofs of Cases 1 and 2, we easily see that $g(n_{p}^{T}, t_{j})\subset g(n_{m+1}^{S}, w)$ for some $w\in A_{m}$. If $T_{j}\cap g(n_{p}^{T}, t_{j})\cap g(n_{p}^{S}, s_{i})\neq\emptyset$, then $T_{j}\cap g(n_{p}^{T}, t_{j})\cap g(n_{m+1}^{S}, s_{i})\neq\emptyset$, hence $g(n_{p}^{T}, t_{j})\subset g(n_{m+1}^{S}, x)\subset g(n_{m}^{S}, x)$ since the family $\{g(n_{m+1}^{S}, w): w\in A_{m}\}$ is disjoint. Now we assume that $T_{j}\cap g(n_{p}^{T}, t_{j})\cap g(n_{p}^{S}, x)=\emptyset$ for any $1\leq j\leq l$. Then there exist $1\leq j\leq l$ and $1\leq h\leq r_{p}^{T} $ such that $y(T, h)\in T_{j}\cap g(n_{p}^{S}, x)$, hence $g(n_{p}^{S}, y(T, h))\subset g(n_{m}^{S}, x)$.

\smallskip
Therefore, it follows from Claims 1 and 2 that $\mathcal{S}(X)$ is a $\gamma$-space.
\end{proof}

 \smallskip
\section{Open questions}
In this section, we pose some open questions on $\mathcal{S}(X)$.

In \cite{LX}, the authors proved that a space $X$ is a semi-stratifiable space if and only if $\mathcal{F}(X)$ is a semi-stratifiable space. Therefore, we have the following question.

\begin{question}
If $X$ is a semi-stratifiable space, is $\mathcal{S}(X)$ a semi-stratifiable space?
\end{question}

In \cite{SF}, the authors proved that $X^{2}$ is monotonically normal if and only if $\mathcal{F}(X)$ is monotonically normal. Therefore, it natural to pose the following question. We conjecture this question is negative.

\begin{question}
If $X^{2}$ is monotonically normal, is $\mathcal{S}(X)$ monotonically normal?
\end{question}

By Theorem~\ref{tt4}, we have the following question.

\begin{question}
If $\mathcal{S}(X)$ is a $\sigma$-space, is $\mathcal{S}(X)$ an $\aleph$-space?
\end{question}

In \cite{LF}, the authors proved that $X$ has a base of countable order (resp. $W_{\delta}$-diagonal, point-regular base) if and only if $\mathcal{F}(X)$ has a base of countable order (resp. $W_{\delta}$-diagonal, point-regular base). Hence, it is interesting to consider the following question.

\begin{question}
If $X$ has a base of countable order (resp. $W_{\delta}$-diagonal, point-regular base), does $\mathcal{S}(X)$ have a base of countable order (resp. $W_{\delta}$-diagonal, point-regular base)?
\end{question}

In \cite{LF}, the authors also proved that $X$ is a Nagata-space if and only if $\mathcal{F}(X)$ is a Nagata-space. Hence, it is natural to pose the following question.

\begin{question}
If $X$ is a Nagata-space, is $\mathcal{S}(X)$ a Nagata-space?
\end{question}

\end{document}